\documentclass[11pt]{article}

\usepackage{hyperref,etex,amsfonts,amssymb,amsmath,amsthm,amscd,euscript,
array,mathrsfs}

\usepackage[all]{xy}
\input{mymacros-3.sty} 
\newcommand{\cO}{\mathcal{O}}
\newcommand{\res}{\vert}
\newcommand{\la}{\langle}
\newcommand{\ra}{\rangle}
\newcommand{\ad}{\mathop{{\rm ad}}\nolimits}

\title{Classification of positive energy representations\\ of the Virasoro group}

\author{Karl-Hermann Neeb, Hadi Salmasian}
\begin{document} 

\maketitle

\begin{abstract}
We give a complete classification of all  positive energy unitary 
representations of the Virasoro group. More precisely, we prove that every such 
representation can be expressed in an essentially  
unique way as a direct integral of irreducible highest weight representations. 
\end{abstract}

\section{Introduction}

Let $\mathrm{Diff}(S^1)_+$ denote the 
group of orientation-preserving smooth diffeomorphisms of the circle. It is well known that
 $\mathrm{Diff}(S^1)_+$ is a smooth Fr\'echet--Lie group
\cite[Ex. 4.4.6]{Ha82}. Let $\mathsf{Vir}$ denote the 
simply connected Virasoro group, that is, the central extension by $\R$ of the simply connected covering of $\mathrm{Diff}(S^1)_+$  
(see Section \ref{sec-ViraS}). Unitary representations of $\mathsf{Vir}$ have been studied for more than three decades. Among unitary representations of 
$\mathsf{Vir}$, the representations of \emph{positive energy} play an important role in physics  \cite{GoKeO} \cite{Kac82}, \cite{segal}.
The positive energy representations are closely related to the unitarizable modules of the Virasoro algebra $\g{vir}$, which is the (unique) central extension by $\C$ of the 
Lie algebra $\g{diff}$  
of complex
vector fields on the circle with finite Fourier 
series. The  Lie bracket of $\g{vir}$ is given by
\begin{equation}
\label{cocycle-for}
\left[e^{imt}\frac{d}{dt},e^{int}\frac{d}{dt}\right]
=i(m-n)e^{(m+n)i}\frac{d}{dt}+\delta_{m,-n}\frac{m(m^2-1)}{12}\kappa. 
\end{equation}
The classification of  irreducible unitarizable highest weight modules of
$\g{vir}$ is by now well known \cite{GoKeO},  \cite{KaRa}, \cite{Langlands}. Furthermore,  every such module of $\g{vir}$
can be integrated to an irreducible positive energy unitary representation $(\pi,\ccH)$ of $\mathsf{Vir}$ (see \cite{GoWa} and \cite{Toledano}). Here positive energy means that the (unbounded self-adjoint) energy operator 
$\dd\pi(-\bd_0)$, where 
$\bd_0:=i\frac{d}{dt}$, has a non-negative spectrum
(see Definition \ref{pos-energy-dfeen}). 

In this paper, we give a complete classification of the (in general reducible) positive energy representations of $\mathsf{Vir}$. Our main result is Theorem 
\ref{thm-MAIN}, in which we prove that every positive energy representation of
$\mathsf{Vir}$ can be expressed in an essentially unique way as a direct integral of irreducible unitary highest weight representations. 
Let us briefly sketch the main ideas of the proof. It is well known that the homogeneous
manifold
$\mathrm{Diff}(S^1)_+/S^1$ 
has an invariant complex structure. (This complex manifold is usually called the Kirillov-Yuriev domain \cite{KiYu}.)
We show that positive energy representations of $\mathsf{Vir}$ can be realized on Hilbert spaces of sections of certain holomorphic vector bundles $\W_\rho$ 
(defined in Section \ref{sec-hvebun}) over the Kirillov--Yuriev domain. 
The construction of $\W_\rho$ is canonical. What is more subtle is to equip $\W_\rho$ with an invariant complex structure. We will achieve this using the  key observation that the $\W_\rho$ are indeed bundles associated to a ``universal'' 
principal bundle $\bS_u$ (defined in Section 
\ref{ssec-univ}). Therefore an invariant complex stucture on $\bS_u$ will induce the suitable invariant complex stuctures on the $\W_\rho$. The complex structure of the principal bundle $\bS_u$ is in turn obtained using the one-parameter family of complex structures on $\mathsf{Vir}$ that is constructed by Lempert \cite{lempert}. After realizing positive energy representations on Hilbert spaces of holomorphic sections, we can use the technique 
of holomorphic induction from \cite{Ne-real}, as well as some new ideas concerning holomorphic reproducing kernels with values in a $C^*$-algebra, to obtain the existence and uniqueness of the desired direct integral decomposition.

We now summarize the content of each of the  sections of this paper. In Section 
\ref{sec-InfiniteDim}, we briefly review the basic theory of infinite dimensional manifolds and Lie groups. 
In Section \ref{sec-cplxstr}, we first study invariant complex structures on a general class of principal bundles that arise from
Lie groups that have a homogeneous space which can be equipped with an invariant complex structure. Next we 
study the properties of $G$-invariant reproducing kernel Hilbert spaces of sections of vector bundles associated to the aforementioned principal bundles. 
For reducible positive energy representations, self-adjoint operators such as the energy operator can have 
a continuous spectrum and therefore one cannot expect to use 
methods that are typically based on weight vectors. To circumvent this issue, we need a technical tool, called  \emph{Arveson spectral subspaces}, which will be presented in Section \ref{sec-arv}.
It is worth mentioning that the setup of Section 3 is quite general and we expect that it will be applicable to other classes of 
Fr\'echet Lie groups. In Section \ref{sec-ViraS}, we focus our attention on $\mathsf{Vir}$, construct  suitable direct integrals of positive energy representations, and prove our main theorem.\\

\noindent\textbf{Acknowledgement.} The second author thanks Professor Laszlo Lempert for illuminating discussions about 
his work on the Virasoro group. During the completion of this project, the second author was supported by an NSERC Discovery Grant and the Emerging 
Field Project ``Quantum Geometry'' of FAU Erlangen--N\"{u}rnberg.

\section{Infinite dimensional manifolds}
\label{sec-InfiniteDim}

Let $E$ and $F$ be locally convex spaces  and $U\sseq E$ be open. Given a map $f:U\to F$, for every $x\in U$ and every $v\in E$ we define the directional derivative $\partial_vf(x)$ by 
\[
\partial_v f(x):=\lim_{t \to 0} \frac{1}{t}(f(x+tv) -f(x))
\]
if the limit exists. 
A continuous map $f : U \to F$ is called \emph{continuously differentiable}, or 
$\cC^1$, if  the maps 
$\partial_vf:U\to F$ 
are defined for all $v\in E$, and the map 
\[
\dd f:U\times E\to F\ ,\ (x,v)\mapsto \dd f(x)(v):=\partial_vf(x)
\]
is continuous.
Similarly, the map $f$ is called $\cC^k$ for some positive integer $k$,  
if the  maps 
\[
\dd^j f:U\times E^j\to F\ ,\ 
(x,v_1,\ldots,v_j)\mapsto 
\dd^j f(x)(v_1,\ldots,v_j):=(\partial_{v_1} \cdots \partial_{v_j}f)(x)
\]
are defined and continuous for every $1\leq j\leq k$. Finally, $f$ is called \emph{smooth}, or $\cC^\infty$, if it is $\cC^k$ for every $k\in\N$.

Based on the above notion of smooth maps we can define smooth locally convex manifolds (see \cite{Ha82} and  \cite{Ne06} for detailed expositions). 
From now on, all manifolds are assumed to be smooth. If $M$ and $N$ are manifolds, then the set of smooth maps from 
 $M$ to  $N$ is denoted by $ \cC^\infty(M,N)$.

Let $M$ and $N$ be manifolds and $f:M\to N$ be a $\cC^1$ map. 
If $p\in M$ and $v\in T_pM$,  then we sometimes write
$\dd_v f$ instead of $\dd f(v)$ or $\dd f(p)(v)$. 
We extend $\dd f:TM\to TN$ to a map between the complexified tangent bundles, that is, 
if $v=v_1+iv_2\in T_p^\C M:=T_pM\otimes \C$, then we set 
$\dd f(v):=\dd f(v_1)+i\dd f(v_2)\in T_{f(p)}^\C N$.

An \emph{almost complex structure} on a manifold $M$ is a smooth bundle 
map $J:TM\to TM$ with the property that
$J_p:=J|^{}_{T_pM}:T_pM\to T_pM$ is an $\R$-linear map satisfying $J_p^2=-I$.
We can write
\[
T^\C_p M=T_p^{1,0}M\oplus T_p^{0,1}M
\]
where $T_p^{1,0}M:=\{v-iJ_pv\,:\,v\in T_pM\}$ and 
$T_p^{0,1}M:=\{v+iJ_pv\,:\,v\in T_pM\}$ are the $\pm i$ eigenspaces of the $\C$-linear extension of $J_p$.
The canonical injection $T_p^{}M\into T^\C_pM$ results in an isomorphism 
$T^{}_pM\cong T^\C_pM/T_p^{0,1}M$ as complex vector spaces. 
A smooth map 
$f:M\to N$ between two almost complex manifolds is called \emph{holomorphic} if 
$J^{}_{f(p)}\circ\dd  f=\dd f\circ J^{}_{p}$.
The set of holomorphic maps from  $M$ to $N$ is denoted by $\mathcal O(M,N)$.

A manifold $M$ is called a 
\emph{complex manifold} if it is locally biholomorphic 
to a complex locally convex space, and the transition maps between the local charts of $M$ are holomorphic. 
For the theory of holomorphic maps between complex manifolds we refer the reader to \cite{Herve}.

An almost complex structure on a real manifold $M$ is called \emph{integrable} if it is induced by a complex manifold structure on $M$. The almost complex structure  is called \emph{formally integrable} if the Nijenhuis tensor $N(\vfx,\vfy):=[\vfx,\vfy]+J[J\vfx,\vfy]+J[\vfx,J\vfy]-[J\vfx,J\vfy]$ vanishes for every two locally defined vector fields $\vfx,\vfy$.
If $M$ is finite dimensional, then by the Newlander--Nirenberg 
Theorem the  notions of integrable and formally integrable coincide. This is no longer true for infinite dimensional manifolds \cite{lempert}.

\subsection{Two lemmas on vector fields}

Let $E$ be a locally convex space and $U\sseq E$ be an open set.  
Every smooth vector field $\vfx$ on $U$ results in a derivation of
$\cC^\infty(U,\C)$
defined by  
$\vfx\cdot f(u):=\dd_{\vfx(u)} f$ for every $u\in U$.
The proof of the following lemma is a straightforward induction.
\begin{lem}
\label{UEY1jYijN}
Let $\vfx_1,\ldots,\vfx_n$ be smooth vector fields on an open set $U\sseq E$. 
Then there exists a family of smooth vector fields 
$
\{\vfy_{p,q}\ :\ 1\leq p\leq n-1\text{ and }1\leq q\leq n_p\}
$
on $U$, depending on $\vfx_1,\ldots,\vfx_n$, 
such that  
\[
\vfx_1\cdot (\vfx_2\cdot (\cdots (\vfx_n\cdot f)\cdots ))(u)=
\dd^n f(\vfx_1(u),\ldots,\vfx_n(u))+\sum_{p=1}^{n-1}
\sum_{q=1}^{n_p}\dd^p f(\vfy_{1,q}(u),\ldots,\vfy_{p,q}(u))
\] 
for every 
$f\in \cC^\infty(U,\C)$ and every $u\in U$.
\end{lem}



\begin{lem} 
\label{X1X2dense}
Let $E$ be a complex locally convex space, $U\sseq E$ be open, $\{\vfx_\alpha\,:\,\alpha\in\mathcal I\}$ be a family of smooth vector fields on $U$,
$f\in\mathcal O(U,\C)$, and $p\in U$. Let $d$ be a positive integer such that
\[
\vfx_{\alpha_1}\cdot (\vfx_{\alpha_2}\cdot(\cdots (\vfx_{\alpha_r}\cdot f)\cdots))(p)=0
\text{ \,for every $1\leq r\leq d$ and every $\alpha_1,\ldots, \alpha_r\in\mathcal I$.}
\] 
If $\spn_\C\{\vfx_\alpha(p)\,:\,\alpha\in\mathcal I\}$ is dense in $T_pU\cong E$,
 then $\dd^rf(v_1,\ldots,v_r)=0$ for every 
 $1\leq r\leq d$ and every $v_1,\ldots,v_r\in T_pU$.
\end{lem} 
\begin{proof}

If $d=1$ then $\dd f(\vfx_\alpha(p))=\vfx_\alpha\cdot f(p)=0$ for every $\alpha\in\mathcal I$. Since the map 
$v\mapsto \dd f(v)$ is $\C$-linear and continuous, 
it follows that $\dd f(v)=0$ for every $v\in T_pU$.
The case $d>1$ is proved by induction on $d$. If $1\leq r\leq d$, then Lemma \ref{UEY1jYijN} implies that  the difference 
\[
\vfx_{\alpha_1}
\cdot(\cdots(\vfx_{\alpha_r}\cdot f)\cdots)(p)
-\dd^r f(\vfx_{\alpha_1}(p),\ldots,\vfx_{\alpha_r}(p))
\]
can be written as a sum of directional derivatives of the form 
$\dd^jf(y_1,\ldots,y_j)=0$ where $1\leq j\leq r-1$ and $y_1,\ldots,y_j\in T_pU$. By the induction hypothesis, all of the latter terms vanish, and therefore 
$\dd^r f(\vfx_{\alpha_1}(p),\ldots,\vfx_{\alpha_r}(p))=0$ for every $\alpha_1,\ldots,\alpha_r\in\mathcal I$. Since 
$\dd^r f$ is $\C$-multilinear and continuous,  it follows that
$\dd^rf(v_1,\ldots,v_r)=0$ for every 
$v_1,\ldots,v_r\in T_pU$.
\end{proof}

\subsection{Infinite dimensional Lie groups and representations }
We assume that all Lie groups are modeled on Fr\'echet spaces and have a smooth exponential map. The exponential map $\g g\to G$, where $\g g:=\Lie(G)$, will be denoted by $x\mapsto e^x$. 
Lie group homomorphisms are always assumed to be smooth.
The universal enveloping algebra of $\gC:=\g g\otimes_\R\C$ will be denoted by 
$\Uu(\gC)$.
By a \emph{complex Lie group} we mean a Lie group that is also a complex manifold with holomorphic multiplication and inversion maps. 
If $G_1$ and $G_2$ are complex Lie groups, by a \emph{homomorphism of complex Lie groups} $f:G_1\to G_2$, we mean a holomorphic map which is also a group homomorphism.

For a Lie group $G$ we define $\ell_g:G\to G$ and $\varrho_g:G\to G$ by $\ell_g(h):=gh$ and $\varrho_g(h):=hg$. If $v\in T^{}_{h}G$ for some 
$h \in G$, then for simplicity we write $g\cdot v$ and $v\cdot g$ instead of $\dd_v^{}\ell_g$ and $\dd_v^{}\varrho_g$.

Let $G$ be a Lie group and $M$ be a
manifold. If $f:U\to M$ is a smooth map, where $U\sseq G$ is open, then we define
\[
L_xf(g):=\dd f(g\cdot x)=\dd_x (f\circ\ell_g)\,\text{ for every }g\in G\text{ and every }x\in \g g:=\Lie(G).
\]
If $M$ has an almost complex structure $J:TM\to TM$, then we define
\begin{equation}
\label{Lx+iyd}
L_{x+iy}f(g):=L_xf(g)+J_{f(g)}^{}L_yf(g)\text{\ \,for every }x,y\in\g g.
\end{equation}

Let $E$ be a complex locally convex space. We denote the vector space  
of continuous linear operators on $E$ 
by $\mathrm B(E)$. 
Let $\mathrm{GL}(E)$ denote the group of 
invertible continuous linear operators on $E$.
By a \emph{strongly continuous} representation of  a Lie group $G$ on $E$ we mean a group homomorphism $\pi:G\to \mathrm{GL}(E)$ 
such that for every $v\in E$, the orbit map $\pi^v:G\to E$ defined by $\pi^v(g):=\pi(g)v$ is continuous.
Given a strongly continuous representation 
$\pi:G\to \mathrm{GL}(E)$,
 we define
\[
\dd\pi(x)v:=\lim_{t\to 0}\frac{1}{t}\left(\pi(e^{tx})v-v\right)
\text{
\ for every $x\in\g g:=\Lie(G)$ and every $v\in E$,
}
\]
if the limit exists.
The group of unitary operators of a complex Hilbert space $\ccH$ will be denoted by $\mathrm U(\ccH)$. 
A strongly continuous representation $\pi:G\to \mathrm U(\ccH)$  
is called a \emph{unitary representation}.
A unitary representation $(\pi,\ccH)$ is called \emph{smooth} if the set 
$
\ccH^\infty:=\{v\in\ccH\ :\ \pi^v\in \cC^\infty(G,\ccH)\}
$
is a dense subspace of $\ccH$. 
We equip $\mathrm B(\ccH)$, $\mathrm{GL}(\ccH)$, and $\mathrm U(\ccH)$ with the operator norm topology. 

\section{Invariant complex structures and unitary representations}
\label{sec-cplxstr}

Let $G$ be a connected Lie group and  $H\sseq G$ be a subgroup which is also a Lie group. 
We assume that 
$H$ is a \emph{split Lie subgroup}, that is, there exists an open subset $\Sigma$ of some locally convex vector space $E$ and a smooth map $\sigma:\Sigma\to G$ such that the map 
\begin{equation}
\label{SHGx)}
\Sigma\times H\to G\ ,\ (z,h)\mapsto \sigma(z)h
\end{equation}
is a diffeomorphism onto an open $\yek$-neighborhood of $G$. 
It follows that $G$ is a smooth principal $H$-bundle over $G/H$ for a suitable manifold structure on $G/H$. 
(In fact the standard argument from the finite dimensional case carries over.) 
Set $M:=G/H$, $\g g:=\Lie(G)$, and $\g h:=\Lie(H)$. 
Let $\pgh:G\to M$ be the canonical projection, and  
set $m_\circ:=\pgh(\yek)$.  

Let $S$ be another connected Lie group and set $\g s:=\Lie(S)$. Let $\gamma: H \to S$ be a Lie group homomorphism, and $\bS := G \times_\gamma S$ denote the associated bundle  corresponding to the smooth action 
\[
H\times S\to S\ ,\ (h,s)\mapsto \gamma(h)s\text{ for every }(h,s)\in H\times S.
\]
Thus, $\bS$ is the quotient of $G\times S$ by the right action of $H$
defined by 
\begin{equation}
\label{gs.h}
(g,s)\cdot h := (gh, \gamma(h^{-1})s) 
\text{ for every }(g,s)\in G\times S\text{ and every }h\in H.
\end{equation}
The associated bundle $\bS$ 
is a $G$-homogeneous principal $S$-bundle over $M$. We will denote the elements of $\bS$ by equivalence classes $[g,s]$ where $g\in G$ and $s\in S$. 
There is a transitive left action of 
$G\times S$ on $\bS$ given by 
\[
(g,s)\cdot [g',s']:=[gg',s's^{-1}]\text{\ \ for every }(g,s)\in G\times S
\text{ and every }[g',s']\in \bS.
\]
Let $\Gamma(\gamma):=\{(h,\gamma(h))\, :\, h\in H\}$ denote the 
graph of $\gamma$. 
The group $\Gamma(\gamma)$ is a split Lie subgroup of $G\times S$ because the map
\[
\Sigma\times S\times H\to G\times S\ ,\ 
(z,s,h)\mapsto (\sigma(z),\yek)(\yek,s)(h,\gamma(h))=(\sigma(z)h,s\gamma(h))
\]
is a diffeomorphism onto an open subset of $G\times S$.
Therefore $G\times S$ is a smooth principal 
$\Gamma(\gamma)$-bundle over 
$(G\times S)/\Gamma(\gamma)$.
The Lie algebra of $\Gamma(\gamma)$ is
$\Gamma(\dd \gamma):=
\{(x,\dd\gamma(x))\, :\, x\in \g h\}\sseq \g g\oplus\g s$ where $\dd\gamma:\g h\to\g s$ is the differential of $\gamma$ at $\yek\in H$. The map
\begin{equation}
\label{PsiSGG}
\Psi:\bS\to (G\times S)/\Gamma(\gamma)\ ,\ [g,s]\mapsto (g,s^{-1})\Gamma(\gamma)
\end{equation}
is $G\times S$-equivariant, and using local charts one can prove that $\Psi$ is a diffeomorphism. 
 We set $\widetilde m_\circ:=\Psi([\yek,\yek])$.
In what follows, it will be  easier to work with $(G\times S)/\Gamma(\gamma)$ than with $\bS$.

\subsection{Formally integrable invariant almost complex structures on $\bS$}

From now on we assume that $S$ is a complex Lie group, and $M$ is a complex manifold with a  
$G$-invariant complex structure that corresponds to a closed 
Lie subalgebra 
$\g q \sseq \gC$ satisfying 
\begin{equation}
  \label{eq:c-str}
\Ad(H)\g q = \g q, \quad \g q \cap \oline{\g q} = \hC, \quad 
\text{ and } \quad \g q + \oline{\g q} = \gC.
\end{equation}
The latter correspondence is given by the canonical isomorphism 
$T_{m_\circ}^{}(M) \cong \g g/\g h \cong \gC/\oline{\g q}$. 

\begin{rmk}
In fact every $G$-invariant formally integrable almost complex structure on $M$ corresponds to a Lie subalgebra $\g q\sseq\gC$ satisfying \eqref{eq:c-str}.
\end{rmk}

\begin{prp}
\label{galcpx}
The formally integrable $G$-invariant almost complex structures on $\bS$ for which the projection map $\pgh_\bS:\bS\to M$ and the right action of $S$ on $\bS$ are holomorphic are in one to one correspondence with continuous Lie algebra homomorphisms $\beta:\oline{\g q}\to \g s$ which satisfy $\beta\big|_{\g h}=\dd\gamma$ and 
\begin{equation}
\label{beAdhxg}
\beta(\Ad(h)x)=
\Ad(\gamma(h))
\beta(x)
\text{\ \ for every }x\in\oline{\g q}\text{ and every }
h\in H.
\end{equation}

\label{prp-cplxGgS}
\end{prp}
\begin{proof}
Suppose that $\bS$ is equipped with an almost complex structure that satisfies the properties of the proposition. 
By the diffeomorphism  
\eqref{PsiSGG},  the almost complex structure of $\bS$ can be transferred to a
$G\times S$-invariant 
formally integrable
almost complex structure on 
$(G\times S)/\Gamma(\gamma)$ with the extra property that the 
projection map $(G\times S)/\Gamma(\gamma)\to M$ and 
the action map 
\begin{equation}
\label{actionSGSGa}
S\times ((G\times S)/\Gamma(\gamma))\to (G\times S)/\Gamma(\gamma)
\ ,\ 
(s',(g,s)\Gamma(\gamma))\mapsto (g,s's)\Gamma(\gamma)
\end{equation}
are holomorphic. 
Let $\g p\sseq \gC\oplus\sC$ denote the Lie subalgebra 
corresponding to the almost complex structure of $(G\times S)/\Gamma(\gamma)$ by the canonical identification 
$T_{\widetilde m_\circ}^{}((G\times S)/\Gamma(\gamma)) \cong (\gC \oplus \sC)/\oline{\g p}$.
Thus, $\g p$ satisfies
\begin{equation}
  \label{eq:c-str2}
\Ad(\Gamma(\gamma))\g p\sseq \g p,\quad
\g p \cap \oline{\g p} = \Gamma(\dd\gamma)_\C^{}, 
\quad \mbox{ and } \quad 
\g p + \oline{\g p} = \gC \oplus \sC.
\end{equation}
Let $\sC = \g s_\C^+\oplus \g s_\C^-$, where  $\g s_\C^+ :=T_\yek^{1,0}S$ and $\g s_\C^- := T_\yek^{0,1}S$.
Holomorphy of the map \eqref{actionSGSGa} implies that 
$\g s_\C^+\sseq \g p$
and $\g s_\C^-\sseq \oline{\g p}$. But if $x\in \oline{\g p}\cap \g s_\C^{+}$ then $x\in\g p\cap\oline{\g p}\cap\sC=\{0\}$, hence  
$
\oline{\g p} \cap \sC = \g s_\C^-
$. 

Let $\pr_{\gC}:\gC\oplus\sC\to \gC$ be the projection onto the first component. Holomorphy of the projection map $(G\times S)/\Gamma(\gamma)\to M$ implies that $\pr_{\gC}(\oline{\g p})=\oline{\g q}$.
Writing
$\oline{\g p}=\oline{\g p}_\circ\oplus\g s_\C^{-}$,
 where 
$\oline{\g p}_\circ:=\oline{\g p}\cap (\gC\oplus \g s_\C^+)$,
we obtain that 
$
\pr_{\gC}(\oline{\g p}_\circ)=\oline{\g q}
$
and
$
\oline{\g p}_\circ\cap\sC=\{0\}.
$
Thus there exists a $\C$-linear map
$\beta^+:\oline{\g q}\to\g s_\C^+$ such that
$
\oline{\g p}_\circ=\Gamma(\beta^+):=\{(x,\beta^+(x))\,:\,x\in \oline{\g q}\}
$.
Since $\oline{\g p}_\circ$ is a Lie subalgebra of $\gC\oplus\sC$, 
the map $\beta^+$ is a Lie algebra homomorphism. 
Moreover, the Closed Graph Theorem for Fr\'echet spaces \cite[Sec. 17]{Tr67} implies that $\beta^+$ is
continuous.

Let $J_\yek:\g s\to\g s$ be the complex structure of $\g s$ and
set $\beta:=j_{\yek}^{-1}\circ \beta^+:\oline{\g q}\to\g s$, where $j_{\yek}^{}:\g s\to\g s_\C^+$ is the canonical identification $j^{}_{\yek}(x):=\frac12 (x-iJ_\yek x)$.
Since $j_{\yek}^{}$ is a Lie algebra homomorphism, so is $\beta$. By the above discussion $\oline{\g p}=\Gamma(\beta^+)\oplus\g s_\C^-$.
Setting 
$\Gamma(\beta):=\{(x,\beta(x))\,:\,x\in\oline{\g q}\}$, we obtain
\begin{equation}
\label{pGbesC}
\oline{\g p}:=\Gamma(\beta)\oplus\g s_\C^-.
\end{equation}
Therefore the relation $\Gamma(\dd\gamma)\sseq \oline{\g p}$ implies that $\beta\big|_{\g h}=\dd\gamma$, and \eqref{beAdhxg} follows from  $\Ad(\Gamma(\gamma))\oline{\g p}=\oline{\g p}$.

Conversely, suppose we are given a $\beta:\oline{\g q}\to\g s$ which satisfies the conditions of the proposition. We define $\oline{\g p}$ as in 
\eqref{pGbesC}. It is straightforward to check that $\oline{\g p}$ is a Lie subalgebra of $\gC\oplus\sC$ which satisfies $\g p+\oline{\g p}=\gC\oplus\sC $.
The relation $\Ad(\Gamma(\gamma)){\g p}\sseq{\g p}$ follows from \eqref{beAdhxg}. 
Next we prove that
$\Gamma(\dd\gamma)^{}_\C=\g p\cap  \oline{\g p}$.
The relation $\beta\big|_{\g h}=\dd\gamma$ implies that 
$\Gamma(\dd\gamma)^{}_\C\sseq \g p\cap \oline{\g p}$, and 
from \eqref{pGbesC} it follows that 
\[
\g p\cap \oline{\g p}\sseq\hC\oplus\sC,
\quad 
\oline{\g p}\cap \sC=\g s_\C^{-},\quad \text{ and } 
\quad
\g p\cap \sC=\g s_\C^+.
\]
Therefore $\g p\cap \oline{\g p}\cap\sC=\g s_\C^-\cap\g s_\C^+=\{0\}$, which implies that 
$\g p\cap \oline{\g p}\sseq \Gamma(\dd\gamma)_\C^{}$. 

We now identify $T_{\widetilde m_\circ}^{}((G\times S)/\Gamma(\gamma))$ 
with $(\gC\oplus\sC)/\oline{\g p}$ to obtain  a formally integrable $G\times S$-invariant 
almost complex structure on $(G\times S)/\Gamma(\gamma)$. By the invariance of almost complex structures, holomorphy of the projection $(G\times S)/\Gamma(\gamma)\to M$ is equivalent to $\pr_{\gC}^{}(\oline{\g p})=\oline{\g q}$, which is a consequence of \eqref{pGbesC}. Similarly, 
holomorphy of the action of $S$ on $(G\times S)/\Gamma(\gamma)$ follows from $\g s_\C^-\sseq \oline{\g p}$.
\end{proof}

\subsection{Holomorphic principal $S$-bundle structures on $\bS$}
\label{sec-hpSb}

Let $\beta:\oline{\g q}\to\g s$ be a continuous Lie algebra homomorphism which satisfies the conditions of Proposition \ref{galcpx}.
Suppose that $\bS$ is equipped with the formally integrable almost complex structure that corresponds to $\beta$.
Our next goal is to determine if $\bS$ 
is indeed a  holomorphic principal 
$S$-bundle. This is achieved in Theorem 
\ref{thm-VGyekL}.

Let $U\sseq M$ be an open set and 
$\bff:U\to \bS$ be a smooth section. Then we can write $\bff$ as  
\[
\bff(gH)=[g,f(g)]\text{\ \ for every }g\in\pgh^{-1}(U), 
\]
where $\pgh:G\to M$ is the canonical projection, 
and $f:\pgh^{-1}(U)\to S$ is a smooth map satisfying
$
f(gh)=\gamma(h)^{-1}f(g)
$  for  $h\in H$ and  
$
g\in \pgh^{-1}(U)
$.
\begin{lem}
\label{lem-schol}
The section $\bff$ is holomorphic if and only if 
\begin{equation}
\label{Lxfgrfg}
L_xf(g)=-\beta(x)\cdot f(g)
\end{equation}
for every $g\in \pgh^{-1}(U)$ and every $x\in\oline{\g q}$.
\end{lem}
\begin{proof} 
Let $\Psi:\bS\to (G\times S)/\Gamma(\gamma)$ be defined as in \eqref{PsiSGG}.
The section
 $\widetilde f$ is holomorphic if and only if  
\begin{equation}
\label{ddpcisile}
\dd(\Psi\circ\bff\circ\ell_g)(\oline{\g q}/\hC)
\sseq 
\dd\ell_{(g,f(g)^{-1})}\left(\oline{\g p}/\Gamma(\dd\gamma)_\C^{}\right)
\text{ for every $g\in \pgh^{-1}(U)$}.
\end{equation}
Let $\varphi:\pgh^{-1}(U)\to S$ be defined by $\varphi(g):=f(g)^{-1}$.
Then \eqref{ddpcisile} is equivalent to 
\begin{equation}
\label{gcxdphixg}
(g\cdot x,\dd\varphi(g\cdot x))\in\dd\ell_{(g,\varphi(g))}(\oline{\g p})\text{ for every }g\in \pgh^{-1}(U)
\text{ and every }x\in\oline{\g q}.
 \end{equation}
From \eqref{pGbesC} and \eqref{gcxdphixg} it follows that if $\widetilde f$ is holomorphic, then 
for every $g\in \pgh^{-1}(U)$ and every $x\in\oline{\g q}$, there exists an 
$s^-\in\g s_\C^-$ such that 
$
\dd\varphi(g\cdot x)=
\varphi(g)\cdot\big(\beta(x)+s^-\big)
$.
The latter equality can be written as
$
\dd f(g\cdot x)=-\big(\beta(x)+s^-\big)\cdot f(g)
$.
Writing $x=x_1+ix_2$, where $x_1,x_2\in \g g$, and $s^-=y+iJ_\yek y$, where $y\in\g s$ and $J_\yek:\g s\to \g s$ is the complex structure of $\g s$, we obtain
\[
\dd f(g\cdot x_1)+i\dd f(g\cdot x_2)=-\beta(x)\cdot f(g)-y\cdot f(g)-iJ_\yek y\cdot f(g).
\]
It follows that $\dd f(g\cdot x_1)=-\beta(x)\cdot f(g)-y\cdot f(g)$ and $\dd f(g\cdot x_2)=-J_\yek y\cdot f(g)=-J_{f(g)}^{}(y\cdot f(g))$, hence
$
L_xf(g)=\dd f(g\cdot x_1)+J_{f(g)}^{}\dd f(g\cdot x_2)=-\beta(x)\cdot f(g)
$.
The proof of the converse is similar.
\end{proof}


\begin{thm}
\label{thm-VGyekL}
Suppose that $\bS$ is equipped with the almost complex structure that corresponds to  
$\beta:\oline{\g q}\to\g s$ according to {\rm Proposition~\ref{galcpx}}. 
Then the following two statements are equivalent.
\begin{itemize}
\item[{\upshape (i)}] $\bS$ is a holomorphic principal $S$-bundle over $M$.
\item[{\upshape (ii)}] There exists an open $\yek$-neighborhood $V\sseq G$ and a smooth map 
$f:V\to S$  satisfying 
\[
L_xf(g)=-\beta(x)\cdot f(g)\text{ \,for every $x\in \oline{\g q}$ and every $g\in V$.}
\]
\end{itemize}
Moreover, when {\upshape (ii)} holds, the complex manifold structure of $\bS$ is uniquely determined by the holomorphic section of $\bS$  given by $gH\mapsto [g,f(g)]$ in some open neighborhood of the base point $m_\circ\in M$.
\end{thm}
\begin{proof}
The implication (i)$\RA$(ii) is an immediate consequence of Lemma \ref{lem-schol}. Next we prove (ii)$\RA$(i).
Recall that $H$ is a split Lie subgroup of $G$, that is, the map $\Sigma\times H\to G$ given in 
\eqref{SHGx)} is a diffeomorphism. After shrinking $V$ we can assume that there exists a connected open subset $\Sigma'\sseq \Sigma$  and a connected open
$\yek$-neighborhood $U_H^{}$ in $H$ such that the map
\[
\Sigma'\times U_H^{}\to G\ ,\ 
(z,h)\mapsto \sigma(z)h
\] 
is a diffeomorphism onto $V$.
Let $g\in \sigma(\Sigma')$ and $\alpha:[0,1]\to H$ be a smooth curve such that $\alpha(0)=\yek$.  
For every $t\in [0,1]$ we have $\alpha'(t)=\alpha(t)\cdot\xi(t)$, 
where $\xi:[0,1]\to \g h$ is a 
smooth map. Then
\begin{align*}
\frac{d}{dt}\Big(\gamma(\alpha(t))f(g\alpha(t))\Big)\Big|_{t=t_\circ} 
&= \gamma(\alpha(t_\circ))\cdot
L_{\xi(t_\circ)}^{}f(g\alpha(t_\circ))
+
(\gamma\circ\alpha)'(t_\circ)\cdot f(g\alpha(t_\circ))\\
&=-\gamma(\alpha(t_\circ))
\cdot
\beta(\xi(t_\circ))
\cdot f(g\alpha(t_\circ))
+
\gamma(\alpha(t_\circ))\cdot
\beta(\xi(t_\circ))\cdot f(g\alpha(t_\circ))=0.
\end{align*}
Therefore $f(gh)=\gamma(h)^{-1}f(g)$ for every $g\in\sigma(\Sigma')$ and every $h\in U_H^{}$. As a result, we obtain an extension of $f$ to a smooth map
$\breve f:\sigma(\Sigma')H\to S$ defined by 
$
\breve f(gh):=\gamma(h)^{-1}f(g)
$ for every $g\in \sigma(\Sigma')$ 
and every $h\in H$.
We now define a smooth section 
\[
\widetilde f:\pgh(V)\to \bS
\ ,\ 
\widetilde f(gH):=[g,\breve f(g)]\text{\ \ for every }g\in \sigma(\Sigma')H.
\]
If $x\in\oline{\g q}$, $g\in\sigma(\Sigma')$, and 
$h\in H$, then 
\begin{align*}
L_x \breve f(gh)
&
=\gamma(h)^{-1}\cdot L_{\Ad(h)x}^{}\breve f(g)=
\gamma(h)^{-1}\cdot L_{\Ad(h)x}^{} f(g)
=-\left(
\beta(x)\cdot\gamma(h)^{-1} \right)\cdot
f(g)
=-\beta(x)\cdot \breve f(gh).
\end{align*}
Therefore Lemma \ref{lem-schol}  
implies that $\widetilde f$ is a holomorphic section of $\bS$.

For every $a\in G$ we set $V_a:=\ell_a(\pgh(V))$ and define
$
\breve f_a:\pgh^{-1}(V_a)\to S$ by $\breve f_a:=\breve f\circ \ell_{a^{-1}}^{}
$.
Note that $L_x \breve f_a(g)=-\beta(x)\cdot \breve f_a(g)$ for every $x\in\oline{\g q}$ and every $g\in \pgh^{-1}(V_a)$.
Therefore the charts
\[
\varphi_a:V_a\times S\to \bS\ ,\ \varphi_a(gH,s):=[g,\breve f_a(g)s]\text{\ \ for every }gH\in V_a\text{ and every }s\in S,
\]
are holomorphic (Lemma~\ref{lem-schol}) 
and turn $\bS$ into a holomorphic principal $S$-bundle over $M$.
\end{proof}
\subsection{The holomorphic vector bundle $\W_\rho$}
\label{sec-hvebun}
Let  $\beta:\oline{\g q}\to\g s$ be a Lie algebra homomorphism that satisfies the conditions of Proposition \ref{galcpx} and Theorem 
\ref{thm-VGyekL}(ii), and suppose that $\bS$ is equipped with the almost complex structure corresponding to $\beta$.
Let
$\ccW$ be a complex Hilbert space, and $\rho:S\to \GL(\ccW)$ be a
homomorphism of complex Lie groups.
Let $\W_\rho:=G\times_{\rho\circ\gamma} \ccW$ denote the vector bundle over $M$ associated to the principal 
$H$-bundle $G\to M$,
corresponding to the smooth action 
\[
H\times\ccW\to \ccW
\ , 
\ (h,w)\mapsto \rho(\gamma(h))w.
\]
That is, $\W_\rho$ is the quotient of $G\times \ccW$ by the right action of $H$ defined by 
\[
(g,w)\cdot h:=(gh,\rho(\gamma(h^{-1}))w)\text{\,\ for every }(g,w)\in G\times \ccW\text{ and every }h\in H.
\] 
We define the associated vector bundle $\bS\times_{\rho}\ccW$ similarly. 
Since the principal $S$-bundle 
$\bS\to M$ and  the action of $S$ on $\ccW$ are holomorphic,  $\bS\times_{\rho}\ccW$ is a holomorphic vector bundle over $M$. 
The complex manifold structure on $\bS\times_{\rho}\ccW$ can be transferred to $\W_\rho$
by the $G$-equivariant diffeomorphism 
\begin{equation*}
\bS\times_{\rho}\ccW
\to \W_\rho
\ ,\ [[g,s],w]\mapsto [g,\rho(s)w]
\text{ \,for }g\in G,\ s\in S,\text{ and }w\in\ccW.
\end{equation*}
From now on, we always assume that $\W_\rho$ is equipped with 
the latter $G$-invariant 
holomorphic vector bundle structure.

Let $\pgh_\bS : \bS \to M$ be the projection of the principal 
$S$-bundle $\bS$ over $M$. 
Then we can write every holomorphic section 
$\sigma : M \to \W_\rho$ as 
$\sigma(\pgh_\bS(p)) = [p, f(p)]$, where 
$f : \bS \to \ccW$ is a holomorphic function. 
Let $\Gamma_\textrm{hol}(\W_\rho)$ denote the space of holomorphic 
sections of the bundle $\W_\rho$. We then obtain a linear bijection 
\begin{equation}
\label{lbij*Ghol}
 \Gamma_{\rm hol}(\W_\rho) \to \cO(\bS,\ccW)_\rho := 
\{ f \in \cO(\bS,\ccW): 
(\forall p \in \bS)(\forall s\in S)\ 
f(p\cdot s) = \rho(s)^{-1} f(p)\}.
\end{equation} 
This bijection intertwines the natural action of $G$ on 
$\Gamma_{\rm hol}(\W_\rho)$ with the action on $\cO(\bS,\ccW)_\rho$ defined by 
\[ (g\cdot f)(p) := f(g^{-1}\cdot p)
\,\text{ for every $p\in\bS$ and every $g\in G$}.\] 

\begin{remark} \label{rem:loctriv}
{\rm Suppose that the bundle $\bS \cong G \times_{\gamma} S\to M$ has a holomorphic 
section $\tilde f : U \to \bS$ over the open subset $U \subseteq M$. 
By Lemma \ref{lem-schol}, $\tilde f$ can be written as  $\tilde f(gH) = [g, f(g)]$, where
$f : \pgh^{-1}(U) \to S$ is a smooth map satisfying 
$f(gh)=\gamma(h)^{-1}f(g)$ for every $g\in G$ and every $h\in H$, and 
$L_x f(g)=-\beta(x)\cdot f(g)$ for every $g\in G$ and every $x\in\oline{\g q}$.
The local trivialization of $\bS\res_U$ results in a holomorphic 
trivialization of $\W_\rho\res_U$ given by  
\begin{equation}
\label{GTticWc}
U\times \ccW\to \W_\rho\res_U\ ,\ (gH,w)\mapsto [g,\rho(f(g))w].
\end{equation}
}\end{remark} 

Sometimes it is also convenient to realize holomorphic sections of 
$\W_\rho$ as functions on $G$. Let 
$\cC^\infty(G,\ccW)_\rho$ denote the vector space of smooth maps $f:G\to\ccW$ satisfying 
\[
f(gh)=\rho(\gamma(h^{-1}))f(g)
\text{ for every }g\in G\text{ and every }h\in H.
\]
The group $G$ acts on $\cC^\infty(G,\ccW)_\rho$ by left translation, that is,
$g_\circ \cdot f(g):=f(g_\circ^{-1}g)$ for $g_\circ,g\in G$. Let $\cC^\infty(G,\ccW)_{\rho,\beta}$ denote the vector space of maps $f\in \cC^\infty(G,\ccW)_\rho$ which satisfy
\begin{equation}
  \label{eq:lx1}
L_xf(g)=-\dd\rho(\beta(x))f(g)\text{ for every }x\in\oline{\g q}\text{ and every }
g\in G.
\end{equation}
Let $\Gamma_\infty(\W_\rho)$ denote the space of smooth sections of $\W_\rho$.
Then we have a linear isomorphism
\begin{equation}
\label{GaMinfC}
\Gamma_\infty(\W_\rho) \xrightarrow{\ \cong\ } 
\cC^\infty(G,\ccW)_\rho\ ,\ 
\widetilde f\mapsto f, 
\end{equation}
where $\widetilde f(gH)=[g,f(g)]$. 
This isomorphism restricts to a $G$-equivariant isomorphism
\begin{equation}
\label{GaMinfC-H}
\Gamma_\mathrm{hol}(\W_\rho) \xrightarrow{\ \cong\ } 
\cC^\infty(G,\ccW)_{\rho,\beta}. 
\end{equation}

\subsection{$G$-invariant Hilbert spaces}

We now present some results on unitary representations of $G$ which can be realized on Hilbert spaces that are embedded in 
$\Gamma_\mathrm{hol}(\W_\rho)$, or equivalently, in
$\cC^\infty(G,\ccW)_{\rho,\beta}$.
 Our main tool is the theory of reproducing kernel Hilbert spaces. 
For a detailed exposition of the theory, see \cite{NeHo}.

Let $\ccH\sseq \cC^\infty(G,\ccW)_{\rho,\beta}$ be a Hilbert space such that, 
for every $g\in G$, the point evaluation $\ev_g^{}:\ccH\to \ccW$, $\ev_g^{} f:=f(g)$ is continuous. Set 
\begin{equation}
\label{RKGtG}
\breve K:G\times G\to \rmB(\ccW)\ ,\ \breve K(g,g'):=\ev_g^{}\circ\ev_{g'}^*.
\end{equation}
where $\ev_g^*:\ccW\to\ccH$ denotes the adjoint of $\ev_g^{}$.
Now let us denote the manifold $\bS$, equipped with the opposite 
complex manifold structure, by $\oline \bS$. 
By the isomorphism \eqref{GaMinfC-H},
we can embed $\ccH$ in
$\cO(\bS,\ccW)_\rho \cong \Gamma_{\rm hol}(W_\rho)$. 
The corresponding point evaluations $\ev_p^{}:\ccH\to \ccW$, $p\in \bS$, 
are continuous. Set 
\begin{equation}
\label{RKMtM}
K:\bS\times \oline \bS\to \rmB(\ccW)
\ ,\  
K(p,p'):=\ev_p^{}\circ\ev_{p'}^*.
\end{equation}
We call the maps \eqref{RKGtG} and \eqref{RKMtM}
the \emph{reproducing kernel} of $\ccH$ on $G$, resp., $\bS$.
It is easy to check that 
\[ K([g,s],[g',s'])=\rho(s^{-1})\breve K(g,g')\rho({s'}^{-1})^*.\]
By a standard fact in the theory of reproducing kernel Hilbert spaces \cite[Lem. I.1.5]{NeHo}, 
the Hilbert space $\ccH$ is uniquely determined by $\breve K$, resp., $K$. 

\begin{dfn}
A map $K:\bS\times \oline \bS\to \rmB(\ccW)$ is called
{\it $(G,\rho)$-invariant} if it is invariant under the 
diagonal $G$-action on $\bS \times \oline \bS$ and satisfies 
\begin{equation}
\label{srhogmK+}
K(p\cdot s,p'\cdot s') = \rho(s^{-1})K(p,p') \rho({s'}^{-1})^*
\,\text{ for every $(p,p')\in \bS\times \oline{\bS}$ and every $s,s' \in \bS$.}
\end{equation}
We call $K$ \emph{positive definite} if 
for every $n\in\N$, every
$w_1,\ldots,w_n\in\ccW$, and every $p_1,\ldots,p_n\in \bS$, we have 
$
\sum_{1\leq i,j\leq n}\lag w_i,K(p_i,p_j)w_j\rag\geq 0$.
\end{dfn}

\begin{dfn}
\label{Ginvdfn}
A Hilbert space $\ccH\sseq \cC^\infty(G,\ccW)_{\rho,\beta}$ with continuous point evaluations $\ev_g^{}$, $g\in G$,  is called \emph{$G$-invariant} if it satisfies the following conditions.
\begin{itemize}
\item[(i)] The map $\ev_\yek^*:\ccW\to \ccH$ is an isometry.
\item[(ii)] $\ccH$ is an invariant subspace of $\cC^\infty(G,\ccW)_{\rho,\beta}$ 
under the action of $G$ by left translation. The so obtained 
$G$-representation $(\pi,\ccH)$ is unitary. 
\end{itemize}
\end{dfn}

\begin{rmk} 
\label{GINVR}
Let $\ccH\sseq \cC^\infty(G,\ccW)_{\rho,\beta}$ 
be a Hilbert space with continuous point evaluations $\ev_g^{}$, $g\in G$ and 
reproducing kernel $\breve K:G\times G\to \rmB(\ccW)$.
Then $\ccH$ is $G$-invariant if and only if 
$\breve K$ satisfies 
\begin{equation}
\label{bK11=1}
\breve K(\yek,\yek)=\yek_\ccW\,\text{ and }\,
\breve K(g_\circ g,g_\circ g')=\breve K(g, g')\quad \mbox{ for every }\quad 
g_\circ,g,g'\in G
\end{equation} 
(cf.\ \cite[Thm.~I.4]{NeHo}). For the sake of completeness, we recall the argument. 

Suppose that $\ccH$ is $G$-invariant. 
Then $\breve K(\yek,\yek)=\ev_\yek^{}\circ\ev_\yek^*=\yek_\ccW$. 
Since the action of $G$ on $\ccH$ by 
$(\pi(g_\circ)f)(g) := f(g_\circ^{-1}g)$ is unitary,  the equality
$\ev_g \circ \pi(g_\circ) = \ev_{g_\circ^{-1}g}$
leads to 
\[ \breve K(g_\circ^{-1}g, g_\circ^{-1} g') 
= \ev_{g_\circ^{-1}g}\circ\ev_{g_\circ^{-1}g'}^*
=  \ev_g \circ \pi(g_\circ)  \pi(g_\circ)^* \circ  \ev_{g'}^* 
=  \ev_g \circ\ev_{g'}^* = \breve K(g,g').\] 

Conversely, suppose that \eqref{bK11=1} holds.
Then
$
\lag \ev_\yek^*w,\ev_\yek^*w'\rag 
=
\lag w,\breve K(\yek,\yek)w'\rag=\lag w,w'\rag$, so that 
$\ev_\yek^*$ is an isometry.
Next we set $\ccH^\circ:=\spn_\C\{\ev_g^*w\,:\,g\in G,\,w\in\ccW\}\sseq \ccH$.
The relation $\lag f,\ev_g^*w\rag =\lag f(g),w\rag$ implies that $\ccH^\circ$ is a dense subspace of $\ccH$ because its orthogonal subspace is trivial. 
The invariance of the kernel $\breve K$ now implies that 
$\ev_{g_\circ^{-1}g}\circ \ev_{g'}^* =  \ev_g \circ \ev_{g_\circ g'}^*,$ 
which leads to 
\begin{equation}
\label{g-cev*gg}
g_\circ \cdot \ev_g^*w=\ev_{g_\circ g}^*w\text{ for every }
g,g_\circ\in G\text{ every }w\in \ccW.
\end{equation}
From \eqref{g-cev*gg} it follows that $\ccH^\circ$ is invariant under the action of $G$ by left translations. Moreover, \eqref{g-cev*gg} 
also implies that
$
\lag g_\circ\cdot \ev_g^*w,g_\circ\cdot\ev_{g'}^*w'\rag=\lag 
\ev_g^*w,\ev_{g'}^*w' 
 \rag 
$ 
for $g_\circ,g,g'\in G$ and $w,w'\in \ccW$. It follows that the action of every $g_\circ\in G$ by left translation on $\ccH^\circ$ extends uniquely to an isometry $\pi(g_\circ):\ccH\to \ccH$. 
That 
$\pi(g_\circ)f(g)=f(g_\circ^{-1}g)$ for every $g,g_\circ\in G$ and every $f\in\ccH$ 
follows from 
\[
\lag \ev_g^{}\pi(g_\circ) f,w\rag
=\lag f,\pi(g_\circ^{-1})\ev_g^*w\rag
=
\lag f,\ev_{g_\circ^{-1}g}^*w\rag=\lag f(g_\circ^{-1}g),w \rag.\]
Finally, from continuity of the matrix coefficients 
$
g_\circ\mapsto \lag \pi(g_\circ)\ev_g^*w,\ev_g^*w\rag=\lag w,\ev^{}_{g_\circ g}(\ev_g^* w)\rag
$ it follows that $(\pi,\ccH)$ is strongly continuous.
\end{rmk}

\begin{prp}
\label{prpEVinBe} 
For a $G$-invariant Hilbert space $\ccH\sseq \cC^\infty(G,\ccW)_{\rho,\beta}$, the following statements are true.
\begin{itemize}
\item[{\upshape (i)}] The unitary $G$-representation 
$(\pi, \ccH)$ is smooth. In fact, 
the subset $\{\ev_g^*w\,:\,g\in G,w\in \ccW\}$ is a total subset 
of $\ccH$ consisting of smooth vectors.
\item[{\upshape (ii)}] The subspace $\ccW^\ev:=\ev^{*}_\yek\ccW\sseq\ccH$ is closed and $H$-invariant. 
\item[{\upshape (iii)}] Set
 $\rho^\ev(h):=\pi(h)\big|_{\ccW^\ev}$ for every $h\in H$. 
Then $(\rho\circ\gamma, \ccW)$ is a unitary representation of $H$ 
and $\ev_\yek^* : \ccW \to \ccW^{\rm ev}$ is a unitary equivalence between the $H$-representations 
$\rho\circ\gamma$ and $\rho^\ev$.

\item[{\upshape (iv)}] The reproducing kernel 
$K:\bS\times \oline \bS\to\rmB(\ccW)$ of $\ccH$ is holomorphic and 
$(G,\rho)$-invariant and satisfies $K(p_\circ, p_\circ)=\yek_\ccW$ 
for $p_\circ: = [\yek,\yek]$. 
\end{itemize} 
\end{prp}

\begin{proof} (i) To prove that $(\pi,\ccH)$ is smooth, it is enough to show that 
$\ev_g^*w\in\ccH^\infty$ for $g\in G$, $w\in \ccW$.
Since $\ev_g^*w\in \cC^\infty(G,\ccW)_{\rho,\beta}$,  the map  
$
G\mapsto\C
$, 
$ 
g_\circ\mapsto\lag\pi(g_\circ)\ev_g^*w,\ev_g^*w\rag=\lag w,\ev_{g_\circ g}^{}\circ\ev_g^*w\rag
$ is smooth. Therefore \cite[Thm 7.2]{nediff} implies that $\ev_g^*w\in\ccH^\infty$.

(ii) Clearly $\ccW^\ev$ is closed because $\ev_\yek^{*}$ is an isometry.  For every $h\in H$, $\ev_h^{}=\rho(\gamma(h^{-1}))\circ \ev_\yek^{}$
%
and therefore $\ev_h^*=\ev_\yek^*\circ \rho(\gamma(h^{-1}))^*$.
Therefore \eqref{g-cev*gg} implies that
$\pi(h)\circ \ev_\yek^*=\ev_\yek^*\circ\rho(\gamma(h^{-1}))^*$ for every $h\in H$.
In particular, $\ccW^\ev$ is $H$-invariant. 

(iii) Since $\pi(h)$ is unitary and $\ev_\yek^*$ is an isometry,
from the equality $\ev_h^*=\ev_\yek^*\circ \rho(\gamma(h^{-1}))^*$ (proved in (ii) above) it follows that
the operator 
$\rho(\gamma(h^{-1}))^*$ is also unitary.
In particular, we obtain $\rho(\gamma(h^{-1}))^*=\rho(\gamma(h))$.
The relation  
$\pi(h)\circ \ev_\yek^*=\ev_\yek^*\circ\rho(\gamma(h))$ now implies that $\rho\circ\gamma$ and $\rho^\ev$ are unitarily equivalent.

(iv) For every $g\in G$, $\ev_g^{}=\ev_\yek^{}\circ\pi(g^{-1})$. It follows that  $\breve K(g,g')=\ev_\yek^{}\circ\pi(g^{-1}g')\circ\ev_\yek^*$
for every $g,g'\in G$. Since $(\pi,\ccH)$ is unitary, we obtain 
$\|\breve K(g,g')\|\leq 1$, and consequently
\[
\|K([g,s],[g',s'])\|\leq \|\rho(s^{-1})\|\cdot \|\rho({s'}^{-1})^*\|.
\] 
This means that the map $K:\bS\times\oline \bS\to \rmB(\ccW)$ is locally bounded. 
Moreover, by \eqref{GaMinfC-H} 
and \eqref{lbij*Ghol}
we can embed $\ccH$ in $\mathcal O_\rho(\bS,\ccW)$, 
so that the general theory of reproducing kernels 
\cite[Lem. A.III.9(iii)]{NeHo} implies that the map
$K:\bS\times\oline \bS\to \rmB(\ccW)$ is holomorphic.
The remaining statements follow from the 
relation between $K$ and $\breve K$, and $\breve K(\yek,\yek)=\yek_\ccW$.
\end{proof}

\begin{prp}
\label{adprp-Gi-ii-iii}
Let $\ccH\sseq \cC^\infty(G,\ccW)_{\rho,\beta}$ be 
a $G$-invariant Hilbert space and $(\pi,\ccH)$ be the unitary representation of $G$  by left translation, defined in {\upshape Definition 
\ref{Ginvdfn}}. Set $\ccW^\ev=\ev_\yek^{*}\ccW\sseq \ccH$.
Then,
\begin{itemize}
\item[{\upshape(i)}] $\pi(G)\ccW^\ev:=\spn_\C\{\pi(g)w\,:\,g\in G,\,w\in\ccW^\ev\}$ is a dense subspace of $\ccH$.
\item[{\upshape(ii)}]
 $\Uu(\gC)\ccW^\ev:=\{\dd\pi(x)w\,:\,
x\in\Uu(\gC)\text,\,w\in\ccW^\ev\}$ is a dense subspace of $\ccH$.

\end{itemize}
 Suppose, in addition, that there exists an 
 $\Ad(H)$-invariant decomposition
$
\gC=\oline{\g n}\oplus\hC\oplus\g n
$ 
such that $\g q=\hC\oplus \g n$, where $\g q$ is the Lie algebra defined in \eqref{eq:c-str}, and 
$\oline{\g n}\sseq\ker(\beta)$.
Then,
\begin{itemize}
\item[{\upshape(iii)}]
$\Uu(\oline{\g n})\ccW^\ev$ is a dense subspace of $\ccH$.
\item[{\upshape(iv)}] $\ccW^\ev=(\ccH^\infty)^{\g n}$, where
$(\ccH^\infty)^{\g n}:=\{v\in \ccH^\infty\, :\, \dd\pi(x)v=0\text{ for every }x\in\g n\}$.
\item[{\upshape(v)}] Assume that every closed $\rho\circ\gamma(H)$-invariant subspace of $\ccW$ is also $\rho(S)$-invariant.
Set $\rho^\ev(h):=\pi(h)\big|_{\ccW^\ev}$ for every $h\in H$. 
Let $\pi(G)'$, resp., $\rho^\ev(H)'$, denote the 
commutants of $\pi(G)$, resp., $\rho^\ev(H)$.
Then the map 
\begin{equation}
\label{RvNag}
\pi(G)'\to\rho^\ev(H)'\ ,\  
T\mapsto T\big|_{\ccW^\ev}
\end{equation}
is an isomorphism of von Neumann algebras.
\end{itemize}
\end{prp}
\begin{proof}
(i)  From Proposition \ref{prpEVinBe}(i) and \eqref{g-cev*gg} we have
%
$\pi(g)\ev_\yek^{*}w=\ev_g^*w$ for $g\in G$ and $w\in \ccW$. 
Therefore $\pi(G)\ccW^\ev=\spn_\C\{\ev_g^*w\,:\,g\in G,\,w\in\ccW\}$, and
Proposition \ref{prpEVinBe}(i) implies that
$\pi(G)\ccW^\ev$ is a dense subset of $\ccH$.

(ii) Given a smooth map $f:G\to\mathscr K$, where $\mathscr K$ is a complex Hilbert space, 
we define 
\[
R_xf(g):=\lim_{t\to 0}\frac 1t (f(e^{-tx}g)-f(g))
\text{ for every $x\in\g g$ and every $g\in G$}.
\]
We extend $R_x$ to $x\in\Uu(\gC)$ as usual, by setting $R_{x+iy}:=R_x+iR_y$ 
and $R_{x_1\cdots x_m}:=R_{x_1}\cdots R_{x_m}$.
Set
\[
\delta:\ccH\to \mathrm{Hom}_\C^{}(\Uu(\gC),\ccW)\ ,\ 
\delta(f)(x):=(R_x f)(\yek)\text{ for every }x\in\Uu(\gC).
\] 
Fix $f\in\ccH\sseq \cC^\infty(G,\ccW)_{\rho,\beta}$ and $w\in\ccW$, 
and set $v:=\ev_\yek^*w\in \ccW^\ev$. 
By Proposition \ref{prpEVinBe}(i), the orbit map $\pi^v:G\to\ccH$, $\pi^v(g):=\pi(g)v$ is smooth,
and \[
\lag f(g),w\rag
=
\lag \ev_\yek^{}\pi(g)^{-1}f,w\rag
=
\lag 
f,\pi(g)\ev_\yek^*w
\rag
=\lag
f,\pi^v(g)
\rag.
\]
By differentiating both sides of the latter equality, we obtain
\begin{equation}
\label{RxphiwW}
\lag R_xf(g),w\rag
=\lag f,R_{\oline{x}}\pi^v(g)\rag
\text{ \,for every }g\in G
\text{ and every }
x\in\Uu(\gC).
\end{equation}
Next we prove that 
\begin{equation}
\label{Rxpiw}
R_x \pi^v(g)=\dd\pi(x^\dagger)\pi^v(g)
\quad \mbox{ for } \quad x\in\Uu(\gC), g\in G,
\end{equation}
where $x\mapsto x^\dagger$ is 
the principal anti-automorphism of $U(\gC)$,
 defined by $x^\dagger :=-x$ for $x\in\gC$. 
It suffices to prove the latter statement for $x=x_1\cdots x_n$ where $x_1,\ldots,x_n\in\g g$.
For $n=1$, the statement is obvious. For $n>1$,  
we prove the statement by induction. 
 For every 
$f_1\in\ccH^\infty$,
\begin{align*}
\lag 
R_{x_1\cdots x_n}\pi^v(g),f_1\rag&
=
\lim_{s\to 0}\frac{1}{s}
\left(
\lag 
R_{x_2\cdots x_n}\pi^v(e^{-sx_1}g),f_1
\rag
-\lag 
R_{x_2\cdots x_n}\pi^v(g)
,f_1\rag
\right)\\
&=-\lag \dd\pi(x_1)\pi^v(g),
\dd\pi(x_2\cdots x_n)f_1\rag
=(-1)^n\lag\dd\pi(x_n\cdots x_1)\pi^v(g),f_1\rag.
\end{align*}
Since $\ccH^\infty$ is a dense subspace of $\ccH$, we obtain 
$R_{x_1\cdots x_n}\pi^v(g)=
(-1)^n\dd\pi(x_n\cdots x_1)\pi^v(g)
$.

Setting $g=\yek$ in \eqref{RxphiwW}, we now obtain 
$
\lag\delta(f)(x),w\rag=\lag f,\dd\pi(\oline x^\dagger)v\rag
$. 
It follows immediately that
\begin{equation}
\label{UgperpWe}
\ker(\delta)=(\Uu(\gC)\ccW^\ev)^\perp.
\end{equation}

To complete the proof of (ii), we need to 
show that $\ker(\delta)=\{0\}$. Fix $f\in\ccH$ such that 
$\delta(f)(x)=0$ for every $x\in\Uu(\gC)$. 
As in Remark~\ref{rem:loctriv}, let 
$\tilde s : U \to \bS, gH \mapsto [g, s(g)]$ 
be a holomorphic section of $\bS\res_U$, where $U$ is an open connected 
neighborhood of $m_\circ$. 
Then the map 
$\varphi_f^{}:U\to \ccW$, $ \varphi_f^{}(gH):=\rho(s(g)^{-1})f(g)$,
is holomorphic, and a computation using the Leibniz rule shows that 
$
R_x(\varphi_f\circ \pgh)(\yek)=0
$ for every $x\in\Uu(\gC)$.
Set 
$\vfx_x(gH):=-\dd\pgh (x\cdot g)$
for every $x\in\g g$ and every $g\in G$. Then $\vfx_x$ is a smooth vector field on $M$, and $R_x(\psi\circ \pgh)=(\vfx_x\cdot \psi)\circ\pgh$ for every $\psi\in \cC^\infty(U,\ccW)$ and every $x\in\g g$. 
Lemma \ref{X1X2dense} implies that, in every local chart, we have 
 $\dd^r\varphi_f^{}(v_1,\ldots,v_r)=0$ for every $r>0$ and every $v_1,\ldots,v_r\in T_{m_\circ}M$. Since $\varphi_f^{}$ is holomorphic it vanishes on $U$. 
It follows that $f = 0$. 

(iii) First we prove that $\ccW^\ev\sseq (\ccH^\infty)^{\g n}$. 
Let $v := \ev^{*}_\yek w$, $w \in \ccW$. Applying \eqref{Rxpiw} with 
$g = \yek$, we obtain for $f \in \ccH$ and $x \in \g n$ the relation 
\[ \la \dd\pi(x)v, f \ra 
=-\la R_x \pi^v(\yek), f\ra 
= - \la w, (R_{\oline x} f)(\yek) \ra
=  \la w, (L_{\oline x} f)(\yek) \ra
= - \la w, \dd\rho(\beta(\oline x)) f(\yek) \ra = 0.\] 
Since $f \in \ccH$ was arbitrary, we obtain 
$\dd\pi(x)v=0$. 

From Proposition \ref{prpEVinBe}(ii)
it now follows that $\pi(H)\ccW^\ev\sseq \ccW^\ev$ and $\dd\pi(x)\ccW^\ev\sseq \ccW^\ev$ for every $x\in{\g q}$.
Consequently,
$\Uu(\gC)\ccW^\ev=\Uu(\oline{\g n})\Uu(\hC)\Uu(\g n)\ccW^\ev=\Uu(\oline{\g n})\ccW^\ev$, and from (ii) we derive that $\Uu(\oline{\g n})\ccW^\ev$ is dense in $\ccH$.

(iv) We have shown in the proof of (iii) that $\ccW^\ev\sseq (\ccH^\infty)^\g n$.
Next we prove that $(\ccH^\infty)^{\g n}\sseq \ccW^\ev$. Let $f_1\in(\ccH^\infty)^{\g n}$. 
Then
$
\lag f_1,\dd\pi(x)f\rag=-\lag \dd\pi(\oline{x})f_1,f\rag=0
$
 for every $f\in\ccH^\infty$
 and every 
$x\in \oline{\g n}$.
In particular, 
$f_1\in\left(\oline{\g n}\Uu(\oline{\g n})\ccW^\ev\right)^\perp$. 
From (iii)  it follows that
$\ccW^\ev+\oline{\g n}\Uu(\oline{\g n})\ccW^\ev$ is dense in $\ccH$.
Moreover, $\ccW^\ev\sseq(\oline{\g n}\Uu(\oline{\g n})\ccW^\ev)^\perp$ because
\[
\lag f,\dd\pi(x)f_2\rag=\lag \dd\pi(\oline{x}^\dagger)f,f_2\rag=0
\text{ for every 
}f,f_2\in\ccW^\ev
\text{ and every }x\in\oline{\g n}\Uu(\oline{\g n}).
\]
This shows that $f_1\in\ccW^\ev$.


(v)
First note that (iv) implies that $\ccW^\ev$ is invariant under 
$\pi(G)'$. Thus, the map \eqref{RvNag} 
is well-defined. 
Now 
let $T\in\pi(G)'$ such that $T\big|_{\ccW^\ev}=0$. Then
$T\big|_{\pi(G)\ccW^\ev}=0$, so that (i) implies that $T=0$. Therefore the map \eqref{RvNag} is an injection.

Next we prove surjectivity of \eqref{RvNag}.
Since every von Neumann algebra is generated by orthogonal projections 
\cite[Chap.~1, \S 1.2]{Dix69},
and images of von Neumann algebras under restriction maps 
are von Neumann algebras 
\cite[Chap.~1, \S 2.1, Prop.~1]{Dix69}, 
it is enough to show that every orthogonal 
projection in $\rho^\ev(H)'$ is contained in the image of the map
\eqref{RvNag}. Let $P \in \rho^\ev(H)'$.
Setting $\ccW^\ev_1 := P\ccW^\ev$ and $\ccW^\ev_2 := (I-P)\ccW^\ev$  yields an 
$H$-invariant orthogonal decomposition 
\begin{equation}
\label{WW1eW2e}
\ccW^\ev = \ccW^\ev_1\oplus\ccW^\ev_2.
\end{equation}
Now set $\ccH_{j}:=\oline{\pi(G)\ccW_j^\ev}$ for 
$j = 1,2$. Surjectivity of \eqref{RvNag} follows  
if we can prove that $\ccH_1\perp\ccH_2$, because then 
$\ccH = \ccH_1 \oplus \ccH_2$ is a $G$-invariant orthogonal 
direct sum, and the 
orthogonal projection $\widetilde P \in \pi(G)'$ onto $\ccH_1$  
 satisfies 
$\widetilde P\big|_{\ccW^\ev} = P$.
We now prove that $\ccH_1\perp \ccH_2$. Because of the relation 
$\pi(h)\circ \ev_\yek^*=\ev_\yek^*\circ\rho(\gamma(h))$ 
(Proposition \ref{prpEVinBe}(iii)),  the decomposition
\eqref{WW1eW2e}
 corresponds to an orthogonal decomposition $\ccW=\ccW_1\oplus\ccW_2$, where each $\ccW_j$, $j=1,2$, is a closed and $\rho\circ\gamma(H)$-invariant subspace. Thus, each $\ccW_j$ is $\rho(S)$-invariant, and 
since $\rho\circ\gamma$ is unitary, it is also $\rho(S)^*$-invariant. 
 Fix $w_j\in\ccW_j$ for $j=1,2$, and let $K:\bS\times \oline \bS\to 
\rmB(\ccW)$ be the reproducing kernel of $\ccH$, as defined in 
 \eqref{RKMtM}. 
By Proposition  \ref{prpEVinBe}(iv),
 the map 
\begin{equation}
\label{MMw1w2K}
\bS\times\oline \bS\to\C\ ,\ 
(p,p')\to \lag K(p,p')w_1,w_2\rag
\end{equation}
is holomorphic.
But if $p = [g,s]$, then 
$K(p,p)=\rho(s^{-1})K(g,g)\rho(s^{-1})^* = \rho(s^{-1})\rho(s^{-1})^*$, and 
therefore the map \eqref{MMw1w2K} vanishes on the diagonal
subset of $\bS\times\oline\bS$. 
Since every holomorphic map on $\bS\times\oline \bS$ 
is uniquely determined by its diagonal values
\cite[Prop. A.III.7]{NeHo},
the map 
\eqref{MMw1w2K} should also vanish for every  $(p,p')\in \bS\times \oline{\bS}$.
Consequently, for every $w_j\in\ccW_j$, $j=1,2$, we have
\[
\lag \pi(g)\ev_\yek^*w_1,\ev_\yek^*w_2\rag
=\lag\ev_g^*w_1,\ev_\yek^*w_2\rag
=\lag w_1,\breve K(g,\yek)w_2\rag  
=\lag w_1,K([g,\yek], [\yek,\yek])w_2\rag=0.
\]
That is, $\pi(G)\ccW_1^\ev\perp\ccW_2^\ev$. This implies that $\ccH_1\perp\ccH_2$.
\end{proof}
\begin{prp}
\label{pEXGi}
Let $K:\bS\times \oline \bS\to\rmB(\ccW)$ be a holomorphic,  
$(G,\rho)$-invariant, positive definite map which satisfies $K(p_\circ, p_\circ) 
= \yek_\ccW$ for $p_\circ := [\yek,\yek]$. 
Then there exists a $G$-invariant Hilbert space $\ccH\sseq 
\cO(\bS,\ccW)_\rho$ whose reproducing kernel, as defined in \eqref{RKMtM},
is equal to $K$.
\end{prp}

\begin{proof}
This follows from standard facts in the theory of reproducing kernels. Therefore we will only sketch the argument. By the isomorphisms \eqref{GaMinfC-H} 
and \eqref{lbij*Ghol}
it suffices to construct $\ccH\sseq \mathcal O(\bS,\ccW)_\rho$.
Set
\[\ccH^\circ:=\spn_\C\{K_{p,w}\,:\,p\in \bS,w\in \ccW\}\sseq \mathcal O(\bS,\ccW)_\rho,
\] 
where $K_{p,w}(p'):=K(p',p)w$. 
We can define a pre-Hilbert structure on $\ccH^\circ$ by the relation \break 
$\lag K_{p,w},K_{p',w'}\rag:=\lag w,K(p,p')w'\rag$.
Let $\ccH$ denote the completion of $\ccH^\circ$. 
From \cite[Prop. I.1.9(iii)]{NeHo} it follows that $\ccH\sseq 
\mathcal O(\bS,\ccW)_\rho$. Continuity of point evaluations of $\ccH$ follows from
the equality $\lag f(p),w\rag=\lag f,K_{p,w}\rag$ for $f\in\ccH$, $p\in \bS$, 
and $w\in\ccW$ (see also \cite[Thm I.1.4]{NeHo}).
The latter equality also implies that $\ev_p^*w=K_{p,w}$, from which 
it follows that $\ev_p\circ\ev_{p'}^*=K(p,p')$.
Next note that 
$
\lag \ev_{p_\circ}^*w,\ev_{p_\circ}^*w'\rag=\lag w,K(p_\circ,p_\circ)w'\rag =\lag w,w'\rag
$, 
that is, $\ev_{p_\circ}^*$ is an isometry. Finally, the condition of 
Definition \ref{Ginvdfn}(ii) follows from $(G,\rho)$-invariance of $K$, which 
implies the $G$-invariance of the kernel $\breve K$ on $G \times G$.
\end{proof}

\begin{prp}
\label{patmston}
For every holomorphic representation $\rho:S\to\GL(\ccW)$, there exists at most one 
$G$-invariant Hilbert space 
$\ccH\sseq \cC^\infty(G,\ccW)_{\rho,\beta}$.
\end{prp}
\begin{proof}
The Hilbert space $\ccH$ is uniquely determined by its reproducing kernel \cite[Lem. I.1.5]{NeHo}. 
By  
\eqref{GaMinfC-H} 
and \eqref{lbij*Ghol}
we can assume $\ccH\sseq \mathcal O(\bS,\ccW)_\rho$, and
it suffices to show that the corresponding $(G,\rho)$-invariant kernel 
$K:\bS\times \oline \bS\to\rmB(\ccW)$ is uniquely determined by $\rho$. 
Since $\ev_{p_\circ}^{}$ is an isometry, we have $K(p_\circ,p_\circ)=\yek_\ccW$.
From $(G,\rho)$-invariance we obtain 
\[ K([g,s], [g,s]) 
= \rho(s^{-1}) K(p_\circ,p_\circ) \rho(s^{-1})^*
= \rho(s^{-1}) \rho(s^{-1})^*\] 
 for every $g\in G,s \in S$.
Therefore $K(p,p)$ is uniquely determined by $\rho$
for every $p= [g,s]\in \bS$. 
From Proposition \ref{prpEVinBe}(iv), 
and the fact that every holomorphic map $\bS\times \oline \bS\to\rmB(\ccW)$ 
is uniquely determined by its diagonal values 
\cite[Prop. A.III.7]{NeHo}, it follows that
$K(p,p')$ is uniquely determined by $\rho$ for every $(p,p')\in \bS\times 
\oline \bS$. 
\end{proof}

\subsection{Arveson spectral subspaces}
\label{sec-arv}
We now recall some results from 
\cite[Appendix A]{Ne-real} on Arveson spectral subspaces. These spectral subspaces were originally introduced in \cite{Ar74}.
Let $V$ be a complete complex locally convex space. A strongly continuous representation
$\alpha:\R\to \mathrm{GL}(V)$ 
is called  \emph{equicontinuous} if, for every 
$\mathbf 0$-neighborhood $U_1\sseq V$, there exists a 
$\mathbf 0$-neighborhood $U_2\sseq V$ such that $\alpha(t)U_2\sseq U_1$ for every $t\in\R$. 
An equicontinuous representation $\alpha$ yields a representation $\alpha: L^1(\R)\to \mathrm B(V)$, given by $\alpha(f)v:=\int_{-\infty}^\infty f(t)\alpha(t)vdt$. 
For every $v\in V$, we set 
\[
I_\alpha(v):=\{f\in L^1(\R)\, :\, \alpha(f)v=0\}
\text{\,\ and\ \,} 
\Spec_\alpha(v):=\{t\in\R\, :\, \widehat f(t)=0\text{ for every }f\in I_\alpha(v)\},
\]
where
$\widehat f(x):=\int_{-\infty}^\infty f(t)e^{itx}dt$.
Similarly, we set 
\[I_\alpha(V):=\{f\in L^1(\R)\,:\,\alpha(f)=0\}\text{ and }
\Spec_\alpha(V):=\{t\in\R\, :\, \widehat f(t)=0\text{ for every }f\in I_\alpha(V)\}.
\]
For every $A\sseq \R$, we define the \emph{Arveson spectral subspace} 
\[
V(A):=V(\alpha,A):=\{v\in V\ :\ \Spec_\alpha(v)\sseq \oline{A}\},
\] where $\oline A$ denotes the closure of $A$. 
From \cite[Rem. A.6]{Ne-real}  it follows that the vector space 
$V(A)$ is a closed subspace of $V$ that is invariant under $\alpha(t)$ for every $t\in\R$
(see also \cite[p.\,225]{Ar74}). 
\begin{lem}
\label{fL1RCpid}
Let $(\pi,\ccH)$ be a smooth unitary representation of $G$. Let $x\in\g g$, and set 
$\pi_x(t):=\pi(e^{tx})$ for every $t\in\R$. Then 
$\pi_x(f)=\widehat{f}(-i\dd\pi(x))$ for every $f\in L^1(\R)$.
\end{lem}

\begin{proof}
Let $\mathfrak B(\R)$ denote the $\sigma$-algebra of Borel subsets of $\R$, and $\mathrm E:\mathfrak B(\R)\to \mathrm B(\ccH)$ denote the spectral measure associated to the self-adjoint operator $-i\dd\pi(x)$.
Since $\widehat f\in L^\infty(\R)$, the operator 
$\widehat f(-i\dd\pi(x))$ is bounded.
Fix $v,v'\in\ccH$ and set $\mathrm E^{v,v'}_\Omega:=\lag \mathrm E_\Omega^{} v,v'\rag$ for every $\Omega\in\mathfrak B(\R)$. Then,
\begin{align*}
\lag \pi_{x}^{}(f)v,v'\rag
&=\int_{-\infty}^\infty f(t)\lag \pi(e^{tx})v,v' \rag dt
=\int_{-\infty}^\infty\int_{-\infty}^\infty 
f(t)e^{itu}dt\, d\mathrm E^{v,v'}(u)=\lag 
\widehat f(-i\dd\pi(x))v,v'\rag. \qedhere
\end{align*} 

\end{proof}

\begin{prp} 
\label{p(i)-(iv)}
 Let 
$(\pi,\ccH)$ be a smooth unitary representation of $G$ and
$x\in\g g$.
Set  $\pi_x(t):=\pi(e^{tx})$ and
 $\alpha(t):=\Ad(e^{tx})\in\GL(\gC)$ 
  for every $t\in\R$.  
Assume that $\alpha$ is equicontinuous. Then the following statements hold.
\begin{itemize}
\item[{\upshape (i)}]  $\pi_x(f)\ccH^\infty\sseq\ccH^\infty$ for every $f\in L^1(\R)$.
Moreover, 
\[
\dd\pi(y)\pi_x(f)v=\int_{-\infty}^\infty f(t)\pi(e^{tx})\dd\pi(\Ad(e^{-tx})y)v\, dt
\text{\,\ for }y\in\gC,\ f\in L^1(\R),\text{ and }v\in\ccH^\infty.
\]

\item[{\upshape (ii)}] 
Set $\ccH^\infty(X):=\ccH(\pi_x,X)\cap \ccH^\infty$
for every  $X\sseq \R$. Then,
\[
\dd\pi\big(\gC(\alpha,A)\big)\ccH^\infty(B)\sseq 
\ccH^\infty(A+B)
\text{ for every $A,B\sseq \R$}.
\]
\item[{\upshape(iii)}] 
Let $\mathrm E:\mathfrak B(\R)\to \mathrm B(\ccH)$ denote the spectral measure associated to the self-adjoint operator $-i\dd\pi(x)$, where
$\mathfrak B(\R)$ is the $\sigma$-algebra of Borel subsets of $\R$.
Let $I\sseq \R$ be an open interval. Then  
$\ccH^\infty\cap \mathrm E_I\ccH$ 
is dense in $\mathrm E_I^{}\ccH$.

\item[{\upshape (iv)}] 
Assume that $x\in\g h$, and that $(\pi,\ccH)$ is obtained,
according to {\rm Proposition \ref{prpEVinBe}}, from 
an embedding $\ccH\into\cC^\infty(G,\ccW)_{\rho,\beta}$
as a $G$-invariant Hilbert space, where $\rho:S\to \GL(\ccW)$ is a homomorphism of complex Lie groups. 
Furthermore, suppose that  there exists an 
$\Ad(H)$-invariant decomposition
$
\gC=\oline{\g n}\oplus\hC\oplus\g n
$ 
such that $\g q=\hC\oplus \g n$, where
$\g q$ is the Lie algebra defined in \eqref{eq:c-str}, and $\oline{\g n}\sseq \ker(\beta)$.
Set $\ccW^\ev:=\ev_\yek^*\ccW$.
Then, 
\[
\Spec_{\pi_x}(\ccH)\sseq \oline{\Spec_{\pi_x}(\ccW^\ev)+\N_0\Spec_{\alpha}(\oline{\g n})},
\] where $\N_0:=\N\cup\{0\}$, and the right hand side
denotes the closure of 
${\Spec_{\pi_x}(\ccW^\ev)+\N_0\Spec_{\alpha}(\oline{\g n})}$.

\end{itemize}
\end{prp}
\begin{proof}
(i) This is a special case of \cite[Thm 2.3]{NSZ14}.

(ii) This is a special case of \cite[Thm 3.1]{NSZ14}.

(iii) Fix $v\in\mathrm E_I\ccH$ and $\eps>0$. 
Choose $w\in \ccH^\infty$ such that $\|v-w\|<\eps$. Choose compact intervals $I_1\sseq I_2\sseq \cdots$ such that $I=\bigcup_{n=1}^\infty I_n$, and a sequence  
$(f_n)_{n \in \N}$ 
in $\cC^\infty_c(\R)$ such that $0\leq f_n(x)\leq 1$, $f_n(x)=1$ for all $x\in I_n$, and  $\mathrm{supp}(f_n)\sseq I$. Let $h_1,h_2,\ldots$ be Schwartz functions such that 
$\widehat h_n=f_n$ for every $n\geq 1$. 
Set $w_1:=f_n(-i\dd \pi(x))w$. 
From (i) and Lemma \ref{fL1RCpid} it follows that 
$w_1=\pi_{x}^{}(h_n)w\in \ccH^\infty$.
Let $\chi_I^{}$ denote the characteristic function of $I$, so that 
 $f_n=\chi_I^{} f_n$. Then 
$\mathrm E_I^{}=\chi_I^{}(-i\dd \pi(x))$, 
and therefore
$
w_1=f_n(-i\dd \pi(x))w
=
\chi_I^{}(-i\dd \pi(x))
f_n(-i\dd \pi(x))w
\in\mathrm E_I^{}\ccH
$.
Moreover, $\|(f_n-\chi_I^{})(-i\dd \pi(\bd_0))w\|\leq \| \mathrm E_{I\bls I_n}w\|$ (see \cite[Thm XII.2.6(c)]{DunSch}), and consequently, 
\begin{align*}
\|v-w_1\|&\leq \|\mathrm E_I^{}v-\mathrm E_I^{}w\|+
\|\mathrm E_Iw-w_1\|\leq \|v-w\|+
\| \mathrm E_{I\bls I_n}w\|.
\end{align*}
Finally, since $I=\bigcup I_n$, we obtain 
$
\lim_{n\to\infty}\| \mathrm E_{I\bls I_n}w\|= 0$.

(iv) Set $A:=\Spec_\alpha(\oline{\g n})$ and $B:=\oline{\Spec_{\pi_x}(\ccW^\ev)+\N_0\Spec_{\alpha}(\oline{\g n})}$. Then $A+B\sseq B$, so that (ii) implies that $\Uu(\oline{\g n})\ccW^\ev
\sseq \ccH(\pi_x,B)$, where
$\Uu(\oline{\g n})\ccW^\ev:=
\{\dd\pi(x)w\,:\,x\in \Uu(\oline{\g n})\text{ and }w\in\ccW^\ev\}$. By 
Proposition \ref{adprp-Gi-ii-iii}(iii), 
$\Uu(\oline{\g n})\ccW^\ev$ is a dense subspace of $\ccH$. Therefore  $\ccH(\pi_x,B)$ is also a dense subspace of $\ccH$. Since $\ccH(\pi_x,B)$ is also closed in $\ccH$, we obtain $\ccH=\ccH(\pi_x,B)$. 
\end{proof}

\section{The Virasoro group}
\label{sec-ViraS}
In this section we use the general results of 
Section \ref{sec-cplxstr} in the special case of the Virasoro group.
 Let $\mathrm{Diff}(S^1)_+$ denote the 
group of orientation-preserving smooth diffeomorphisms of the circle. As shown in 
\cite[Ex. 4.4.6]{Ha82}, $\mathrm{Diff}(S^1)_+$ is a smooth Fr\'echet--Lie group. 
Note that $\pi_1(\mathrm{Diff}(S^1)_+)\cong \Z$, and the simply connected covering of $\mathrm{Diff}(S^1)_+$ is
\[
\Diff_+:=\{\phi\in \cC^\infty(\R,\R)\,\ :\,\phi(t+2\pi)=\phi(t)+2\pi,\text{ and }\phi'(t)>0\text{ for every }t\in\R\}.
\] 

From now on $G$ will denote the simply connected Virasoro group $\mathsf{Vir}$, that is,  the central extension of $\Diff_+$ by $\R$ whose multiplication is defined by
\[
(\phi,a)(\psi,b):=(\phi\circ\psi,a+b+\cycl(\phi,\psi))\textrm{ \,for every }\phi,\psi\in\Diff_+ 
\text{ and every }a,b\in\R,
\] 
where $\cycl(\phi,\psi)$ is the Bott cocycle $
\cycl(\phi,\psi):=\frac{1}{2}\int_0^{2\pi}\log (\phi\circ\psi)'d\log\psi'$.
The center of $G$ is 
\[
Z(G):=\{(\phi_n,a)\in G\,:\, n\in \Z\text{ and }a\in\R\}\cong \Z\times \R,
\]
where $\phi_n(x):=x+2\pi n$ for every $x\in \R$.
Let $\g g:=\Lie(G)$. We can represent every element of $x\in\gC$ by 
an infinite series $x=z\vecc+\sum_{n\in\Z}a_n \bd_n$, where $z,a_n\in\C$ and
\[
\|x\|_k:=|z|+\sum_{n\in \Z}(|n|^k+1)|a_n|<\infty
\text{ for every }k \in \N.
\]
Here $\bd_n:=ie^{int}\frac{d}{dt}$, and $\vecc\in Z(\g g)$ is chosen such that the Lie bracket of $\gC$ is given by
\[
[\bd_m,\bd_n]=(n-m)\bd_{m+n}+2\pi im^3 \delta^{}_{m+n,0}\vecc\text{ \,for every }m,n\in\Z.
\]
Note that in the above formula, the choice of the cocycle is different from 
\eqref{cocycle-for}
which is more commonly used in the literature, but this will not have a significant effect on our presentation. 
Observe that 
$\g g=\{x\in\gC\, :\, \oline{x}=x\}$, where the complex conjugation $x\mapsto \oline{x}$ is uniquely specified by $\oline{\vecc}=\vecc$ and $\oline{\bd}_n=-\bd_{-n}$.
Moreover, $e^{-2\pi n i\bd_0}=(\phi_n,0)\in Z(G)$ for every $n\in\Z$. 

From now on, let $H$ denote the connected Lie subgroup of $G$ which corresponds to the 
Lie subalgebra 
$
\g h:=\R\vecc\oplus \R i\bd_0\sseq \g g
$.
By a well-known result of Kirillov and Yuriev \cite{KiYu}, the homogeneous manifold $M:=G/H$ is 
a complex manifold 
(diffeomorphic to a complex Fr\'echet domain) with a  $G$-invariant 
complex structure  
which corresponds to the Lie subalgebra 
$\g q=\g h_\C^{}\oplus\g n\sseq \gC$,  where
\[
\g q:=\left\{x\in\gC\ :\ x=z\vecc+\sum_{n=0}^\infty a_n\bd_n,\ \ 
z,a_n\in\C
\right\}
\text{ \,and \,}
\g n:=\left\{
x\in\gC\ :\ x=\sum_{n=1}^\infty a_{n}\bd_{n},\ \  a_n\in\C
\right\}.
\]

\subsection{Left invariant complex structures on the Virasoro group}

Following \cite{lempert}, we define a one-parameter family of (integrable) left-invariant almost complex structures on $G$. Namely, for every  $\tau\in\C\setminus\R$ we write 
$
T^\C_\yek G=\g g^{1,0}_\tau\oplus \g g^{0,1}_\tau,
$ 
where
\[
\g g^{1,0}_\tau:=\left\{
x\in\gC\,:\ x=\tau a_0\vecc+
\sum_{n=0}^\infty a_n\bd_n
\right\}\,\text{ and }\,
\g g^{0,1}_\tau=\left\{
x\in\gC\,:\,x=\oline\tau a_0\vecc+
\sum_{n=0}^{\infty} a_{-n}\bd_{-n}
\right\}.
\]
The canonical identification 
$T^{}_\yek G\cong T_\yek^\C G/\g g^{0,1}_\tau$ yields a  left-invariant almost complex structure on $G$.
This almost complex structure 
induces an almost complex structure on $H$ corresponding to the direct sum decomposition 
$\hC=\g h^{1,0}_\tau\oplus\g h^{0,1}_\tau$,
where $\g h^{1,0}_\tau:=\hC\cap \g g^{1,0}_\tau$ and 
$\g h^{0,1}_\tau:=\hC\cap \g g^{0,1}_\tau$.
From now on, to specify the almost complex structures chosen on $G$ and $H$, we write $G_\tau$ and $H_\tau$.
Note that $H_\tau\cong \C$ as complex Lie groups.

\begin{prp}
\label{kiriYur}
For every $\tau\in\C\bls \R$, the almost complex structure of $G_\tau$ is integrable. The complex manifold $G_\tau$ is a 
holomorphic principal $H_\tau$-bundle  over $M$ with a global trivialization.  
\end{prp}
\begin{proof}
This is proved in \cite[Thm 4.2]{lempert}. Actually Lempert gives the proof for $G/D$ where $D\sseq Z(G)$ is the discrete subgroup of $G$ generated by $(\phi_1,0)$.
Since the pullback of an integrable complex structure to a covering manifold 
is also integrable, it follows that 
Lempert's result also holds for the covering group $G$.
\end{proof}

\subsection{The universal  bundle $\bS_u$ and the associated vector bundles $\W_\rho$}
\label{ssec-univ}
Let $H_\C\cong \C^2$ denote the complexification of $H$ and $\gamma_u:H\to H_\C$ denote the canonical embedding, so that $\dd\gamma_u:\g h\into\g h_\C^{}$ is the canonical injection. 
Set $\bS_u:=G\times_{\gamma_u^{}} H_\C$, where the right hand side is the quotient of the action defined in \eqref{gs.h}.
We equip $\bS_u$ with the almost complex structure that corresponds, 
according to Proposition \ref{galcpx},
to the canonical projection $\beta_u:\oline{\g q}\to\g h_\C^{}$.

\begin{prp}
\label{acpGgambet}
$\bS_u$ is a holomorphic principal $H_\C$-bundle over $M$ with a global trivialization.
\end{prp}

\begin{proof}
Fix $\tau\in\C\bls\R$ and let  $\beta_\tau:\oline{\g q}\to \g h_\tau$ be defined by $\beta_\tau:=\mathsf p_\tau\circ \beta_u$, where $\mathsf p_\tau:\hC\to\hC/\g h_{\tau}^{0,1}\cong \g h_\tau$ is the canonical projection. Set 
$\bS_\tau:=G\times_{\gamma_\tau^{}}H_\tau$, 
where $\gamma_\tau^{}:H\to H_\tau$ is the identity map. We equip $\bS_\tau$ with the almost complex structure that corresponds, according to Proposition \ref{galcpx}, to $\beta_\tau$. 
Let $\widetilde s_\tau:M\to G_\tau$ be a holomorphic global section for the trivial bundle $G_\tau\to M$, and set 
\[
s_\tau:G\to H_\tau\ ,\ s_\tau(g):=g^{-1}\widetilde s_\tau(gH).
\]
The map $G_\tau\to \bS_\tau$ given by $g\mapsto [g,\yek]$ is a holomorphic diffeomorphism. Therefore Lemma 
\ref{lem-schol}
implies that 
\begin{equation}
\label{L-xstaug)}
L_x s_\tau(g)=-\beta_\tau(x)\cdot s_\tau(g)
\text{ for every }x\in\oline{\g q}
\text{ and every }g\in G.
\end{equation}
For every $\tau\in\C\bls\R$, let $H_\tau^{0,1}$ denote the Lie subgroup of $H_\C$ corresponding to the  Lie subalgebra $\g h_\tau^{0,1}$. 
Now fix $\tau_1,\tau_2\in\C\setminus\R$ such that $\tau_1\neq\tau_2$.
The map
$H_\C\to H_{\tau_1^{}}^{}\times H_{\tau_2^{}}^{}$, 
$h\mapsto (h H_{\tau_1^{}}^{0,1},h H_{\tau_2^{}}^{0,1})$ 
is an isomorphism of complex Lie groups.
The map
\[
s:G\to H_\C\ ,\ 
s(g):=(s_{\tau_1}^{}(g),s_{\tau_2}^{}(g))
\in H_{\tau_1^{}}\times H_{\tau_2^{}}\cong H_\C
\]
is smooth, and from \eqref{L-xstaug)} it follows that
$L_xs(g)=-\beta_u(x)\cdot s(g)$ for every $x\in\oline{\g q}$ and every $g\in G$.  
Therefore the proposition follows from Theorem \ref{thm-VGyekL}.
\end{proof}
We now need the following important observation.
Let $\rho:H\to \Uu(\ccW)$ be a Lie group homomorphism.  
From \cite[Thm III.1.5]{Ne06} it follows that 
$\rho$ can be extended uniquely to a homomorphism of 
complex Lie groups $\rho:H_\C\to\GL(\ccW)$.
Therefore the construction of Section \ref{sec-hvebun}
makes  $\W_\rho:=G\times_{\rho\circ\gamma_u}\ccW$
a holomorphic vector bundle over $M$. From now on, we assume that $\W_\rho$ is equipped with the latter holomorphic vector bundle structure.




\subsection{Positive energy representations}
Let $(\pi,\ccH)$ be a smooth unitary representation of 
the simply connected 
Virasoro group $G$. The unbounded self-adjoint operator
$-\dd\pi(\bd_0)$ is usually called the 
\emph{energy operator} of $(\pi,\ccH)$.
\begin{dfn}
\label{pos-energy-dfeen}
A smooth unitary representation  $(\pi,\ccH)$ 
of $G$ is said to be a  
\emph{positive energy}  representation if 
$
\mathrm{Spec}(-\dd\pi(\bd_0))\sseq\R^+\cup\{0\}
$.
\end{dfn}

\begin{prp}
\label{prpPoScyq}
Let $(\pi,\ccH)$ be a positive energy representation of $G$, and set
$\ccW^\pi:=\oline{(\ccH^\infty)^{\g n}}$. Then
$\ccW^\pi$ is $H$-invariant, and  $\pi(G)\ccW^\pi:=\spn_\C\{\pi(g)w\,:\,g\in G\text{ and }w\in\ccW^\pi\}$ is dense in $\ccH$.
\end{prp}
\begin{proof}
The $H$-invariance of $\ccW^\pi$ follows from
the relation $\dd\pi(x)\pi(h)v=\pi(h)\dd\pi(\Ad(h^{-1})x)v$ and $\Ad(H)$-invariance of $\g n$. Next we prove
that $\pi(G)\ccW^\pi$ is dense.
Set $\ccH':=(\pi(G)\ccW^\pi)^\perp$. Then $\ccH'$ is a
closed and $G$-invariant subspace of $\ccH$, and the representation of $G$ on $\ccH'$ is of positive energy. Now set
\[
a:=\inf\mathrm{Spec}\left(-\dd\pi(\bd_0)\big|_{\ccH'}^{}\right)\text{ and }\textstyle I:=(a-\frac12,a+\frac12)\sseq \R.
\] 
Let $\mathrm E:\mathfrak B(\R)\to\mathrm B(\ccH')$ be the spectral measure associated to the self-adjoint 
operator $-\dd\pi(\bd_0)\big|_{\ccH'}^{}$. 
By Proposition \ref{p(i)-(iv)}(iii), there exists a nonzero  
$v\in\ccH^\infty\cap \mathrm E_I\ccH'$.  
Now set 
$x:=-i\bd_0$, and define
$\alpha(t):=\Ad(e^{tx})\in\GL(\gC)$ and 
$\pi_x(t):=\pi(e^{tx})$ 
for every $t\in\R$. 
It is straightforward to check that $v\in\ccH^\infty(I)$ and that 
$\alpha$ is equicontinuous. In view of 
\[  \alpha(t)\bd_n = e^{t \ad(-i \bd_0)} \bd_n = e^{-itn}\bd_n \quad \mbox{ and }\quad \alpha(f)\bd_n = \hat f(-n)\bd_n\] 
we obtain with Proposition~\ref{p(i)-(iv)}(ii) that 
\[
\textstyle\dd\pi(\bd_n)v\in
\ccH(\pi_x,[a-n-\frac12,a-n+\frac12])=
\mathrm E_{[a-n-\frac12,a-n+\frac12]}\ccH=\{0\}
\text{\,\ for every $n>0$.}
\] 
Consequently, $v\in\ccW^\pi$, which is a contradiction.
\end{proof}
Let $(\pi,\ccH)$ be a positive energy representation
of $G$.
We fix some notation which will appear in the rest of this article. Set $\ccW^\pi:=\oline{(\ccH^\infty)^{\g n}}\sseq \ccH$. Note that $\ccW^\pi$ is $H$-invariant. Now let $(\rho^\pi,\ccW^\pi)$ denote the unitary representation of $H$ that is obtained from the restriction of $\pi$. That is, $\rho^\pi(h):=\pi(h)\big|_{\ccW^\pi_{}}^{}$ for every $h\in H$.
Let 
\begin{equation}
\label{dfn-Epi}
\mathrm E^\pi_{}:\mathfrak B(\widehat{H})
\to
\rmB(\ccW^\pi)
\end{equation}
denote the spectral measure corresponding to 
$(\rho^\pi,\ccW^\pi)$, in the sense of Stone's Theorem \cite[Thm X.2.1]{FeDo88}. 
Here $\widehat{H}$ denotes the dual of $H$, and
$\mathfrak B(\widehat{H})$ denotes the $\sigma$-algebra of Borel 
subsets of $\widehat{H}$.


\begin{dfn}
A positive energy representation $(\pi,\ccH)$ of $G$ 
is said to be \emph{of bounded type} if 
$\rho^\pi:H\to\Uu(\ccW^\pi)$ is a  
Lie group homomorphism. 
\end{dfn}
\begin{rmk}
\label{rmk:i-iv}
For a unitary representation $(\rho,\ccW)$ of $H$, 
the following statements are equivalent.
\begin{itemize}
\item[(i)] $\rho:H\to \Uu(\ccW)$ is continuous with respect to the operator norm.
\item[(ii)] $\rho:H\to \Uu(\ccW)$ is a  Lie group homomorphism.
\item[(iii)] The spectral measure corresponding to $\rho$ has compact support.
\item[(iv)] $\rho$ has a unique extension to a homomorphism of complex Lie groups $\rho:H_\C\to\GL(\ccW)$.
\end{itemize}
\end{rmk}

The next proposition is a variation of \cite[Thm 2.17]{Ne-real}.
\begin{prp}
\label{rglVgqendv}
Let $(\pi,\ccH)$ be a positive energy representation of $G$, 
 $\ccW\sseq \ccH$ be a closed $H$-invariant subspace, and $P_\ccW:\ccH\to\ccW$ be the orthogonal projection on $\ccW$. 
Set  $\rho(h):=\pi(h)\big|_{\ccW}$ for every $h\in H$, so that $(\rho,\ccW)$ is a unitary representation of $H$. 
 Assume the following statements.
\begin{itemize}
\item[{\upshape (i)}]  $\rho:H\to\Uu(\ccW)$ is a Lie group homomorphism.
\item[{\upshape(ii)}] $\pi(G)\ccW:=\spn_\C\{\pi(g)w\,:\,g\in G,\, w\in\ccW\}$
is dense in 
$\ccH$, and $(\ccH^\infty)^{\g n}\cap \ccW$ is dense in $\ccW$.
\end{itemize}
For the unique holomorphic extension to a homomorphism $\rho:H_\C\to\GL(\ccW)$ 
of complex Lie groups, the map
\[
\ccH\to \cC(G,\ccW)\ ,\ v\mapsto f_v\text{\,\ where }f_v(g):=P_\ccW\pi(g)^{-1}v
\]
is an embedding of $(\pi,\ccH)$ in $\cC^\infty(G,\ccW)_{\rho,\beta_u}$
as a $G$-invariant Hilbert space and $\ccW= (\ccH^\infty)^{\g n}$.
\end{prp}

\begin{proof}
If $v\in\ccH^\infty$, then 
for every $x\in\oline{\g q}$ and every $w\in (\ccH^\infty)^{\g n}\cap\ccW$ we have
\begin{align*}
\lag P_\ccW\dd\pi(x)\pi(g)^{-1}v,w\rag 
&=\lag \dd\pi(x)\pi(g)^{-1}v,w\rag
=-\lag \pi(g)^{-1}v,\dd\pi(\oline x)w\rag\\
&=\lag P_\ccW\pi(g)^{-1}v,\dd\rho(\beta_u(x))^*w\rag 
=\lag \dd\rho(\beta_u(x))P_\ccW\pi(g)^{-1}v,w\rag. 
\end{align*}
If follows that 
$
L_xf_v(g)
=
-P_\ccW\dd\pi(x)\pi(g)^{-1}v
=
-\dd\rho(\beta_u(x))f_v(g)
$
for every 
$x\in\oline{\g q}$. Consequently, 
$f_v\in \cC^\infty(G,\ccW)_{\rho,\beta_u}$ for every $v\in\ccH^\infty$.

Let $\Gamma_\text{ct}(\W_\rho)$ denote the space of continuous sections of $\W_\rho$, equipped with the compact-open topology. By \cite[Cor. III.12]{neeb-01}, the space $\Gamma_\textrm{hol}(\W_\rho)$ of holomorphic sections of $\W_\rho$ is a closed subspace of $\Gamma_\text{ct}(\W_\rho)$.
For every $v\in \ccH$, we set 
\[
\widetilde f_v:M\to \W_\rho\ ,\ 
\widetilde f_v(gH):=[g,f_v(g)]\text{ for every }g\in G.
\]
Then $\widetilde f_v\in\Gamma_\textrm{ct}(\W_\rho)$, and the map
$\bfJ:\ccH\to\Gamma_\textrm{ct}(\W_\rho)$ defined by $\bfJ(v):=\widetilde f_v$
is continuous. We have shown above that 
$f_v\in \cC^\infty(G,\ccW)_{\rho,\beta_u}$ for every $v\in\ccH^\infty$. Therefore from the isomorphism \eqref{GaMinfC-H} it follows that $\bfJ(\ccH^\infty)\sseq\Gamma_\text{hol}(\W_\rho)$. 
Since $\ccH^\infty$ is a dense subspace of $\ccH$, we obtain that $\bfJ(\ccH)\sseq \Gamma_\text{hol}(\W_\rho)$, that is, 
$f_v\in \cC^\infty(G,\ccW)_{\rho,\beta_u}$ for every $v\in \ccH$.

If $f_v=0$ for some $v\in\ccH$, then  $v\in (\pi(G)\ccW)^\perp$, that is, $v=0$. Therefore 
the map $v\mapsto f_v$ is an injection.
Checking the $G$-equivariance relation $f_{\pi(g_\circ)v}(g)=f_v(g_\circ^{-1}g)$, and continuity of point evaluations $\ev_g^*v:=f_v(g)$ for $g\in G$, are straightforward. 
The reproducing kernel of $\ccH$, as defined in 
\eqref{RKGtG}, is given by 
$\breve K(g,g')=P_\ccW\pi(g)^{-1}\pi(g')P_\ccW^*$, and therefore it satisfies \eqref{bK11=1}. 
Remark \ref{GINVR} now implies that 
$\ccH$ is $G$-invariant. 
The equality $\ccW=(\ccH^\infty)^{\g n}$ follows from Proposition
\ref{adprp-Gi-ii-iii}(iv).
\end{proof}

\begin{prp}
\label{pipij}
Let $(\pi,\ccH)$ be a positive energy representation of $G$. Then there exists
a sequence $(\pi_j)_{j =1}^n$, 
where $n\in \N\cup\{\infty\}$, of positive energy representations of bounded type such that 
$\pi=\oplus_{j=1}^n\pi_j$,
and 
the spectral measures $\rmE^{\pi_j}_{}$, as defined in
\eqref{dfn-Epi}, are mutually singular.
\end{prp}
\begin{proof}
Let $C_\circ\cong\R$ denote the one-parameter subgroup $e^{\R\vecc}\sseq G$,
and let $\rmE^\vecc:\mathfrak B(\R)\to \rmB(\ccH)$ denote the spectral measure
of the self-adjoint operator $\dd\pi(i\vecc)$. 
Then 
$\ccH=\oplus_{n\in\Z}\ccH^n$, where $\ccH^n:=\rmE^\vecc_{[n,n+1)}\ccH$. Since
$e^{\R\vecc}\sseq Z(G)$, every $\ccH^n$ is $G$-invariant. 
Therefore, from now on, we can assume that $\ccH=\ccH^n$ for some $n\in\Z$.

Now let $D_\circ\cong \R$ be the one-parameter subgroup $e^{\R i\bd_0}\sseq G$ 
and let $\rmE^{\bd_0}:\mathfrak B(\R)\to\rmB(\ccH)$
denote  the spectral measure of
the self-adjoint operator $-\dd\pi(\bd_0)$.
%
%
Set $a_1=\inf(\supp(\rmE^{\bd_0}))$.
Note that $a_1\geq 0$, because  $(\pi,\ccH)$ is a positive energy representation.
Set $I_1:=(a_1-\frac12,a_1+\frac12)\sseq\R$ and
$\ccW_1:=\mathrm E_{I_1}^{\bd_0}\ccH$. Then $\ccW_1$ is $H$-invariant, and
from Proposition \ref{p(i)-(iv)}(iii) it follows that
$
\ccH^\infty\cap\ccW_1$ is a dense subspace of $\ccW_1$. On the other hand, setting $x:=-i\bd_0$ in Proposition \ref{p(i)-(iv)}(ii) implies that
$\ccH^\infty\cap \ccW_1\sseq(\ccH^\infty)^\g n$ 
(cf.\ the proof of Proposition~\ref{prpPoScyq}). 
Consequently, $(\ccH^\infty)^\g n\cap \ccW_1$ is 
a dense subspace of $\ccW_1$. 

Now let $(\rho_1^{},\ccW_1)$ denote the representation of $H$ defined by
$\rho_1^{}(h):=\pi(h)\big|_{\ccW_1}$ for $h\in H$. 
The above arguments imply that the spectral measure of $\rho_1^{}$ is compactly supported, and therefore $\rho_1^{}:H\to\Uu(\ccW_1)$ is a smooth Lie group homomorphism.
From 
\cite[Thm III.1.5]{Ne06} it follows that 
there exists a unique extension of $\rho_1^{}$ 
 to a homomorphism of complex Lie groups $\rho_1^{}:H_\C\to\GL(\ccW_1)$. 
Set $\ccH_1:=\oline{\pi(G)\ccW_1}$
where
$\pi(G)\ccW_1:=\spn_\C\{\pi(g)w\,:\,g\in G,\,w\in\ccW_1\}$,
 and let $(\pi_1,\ccH_1)$ be the representation of $G$  defined by $\pi_1(g):=\pi(g)\big|_{\ccH_1^{}}$ for $g\in G$. 
By Proposition \ref{rglVgqendv}, the representation $(\pi_1,\ccH_1)$ can be realized as a $G$-invariant Hilbert space in 
$ \cC^\infty(G,\ccW_1)_{\rho_1^{},\beta_u}$, and
$\ccW_1=(\ccH^\infty_1)^{\g n}$.
Consequently, $(\pi_1,\ccH_1)$ is of bounded type.

We now repeat the above process with the representation of $G$ on $\ccH_1^\perp$. That is, we set 
$a_2=\inf(\supp(\rmE^{\bd_0}\big|_{\ccH_1^\perp}))$, $I_2=(a_2-\frac12,a_2+\frac12)$, $\ccW_2=\mathrm E^{\bd_0}_{I_2}\ccH_1^\perp$, $\ccH_2=\oline{\pi(G)\ccW_2}$, and so on. Note that it is possible that $I_1\cap I_2\neq\varnothing$, but the spectral measures on $\ccW_1$ and $\ccW_2$ will be mutually singular, because $
\ccW_2=\mathrm E^{\bd_0}_{I_2^{}}\ccH_1^\perp=
\mathrm E_{I_2}^{\bd_0}\mathrm E_{\R\bls I_1}^{\bd_0}\ccH=
\mathrm E_{I_2\bls I_1}\ccH
$.
\end{proof}

\subsection{The $H$-spectrum of a positive energy representation}
\label{subsecTspec}

Let $(\pi_\circ,\ccH_\circ)$ be an
irreducible positive energy representation of $G$. Then 
from \cite[Prop. 4.11]{NeConfl} it follows that
$(\pi_\circ,\ccH_\circ)$ is a \emph{highest weight} representation, that is, $\mathrm{Spec}(-\dd\pi_\circ(\bd_0))$ is  
discrete  and there is a unique eigenvector of $-\dd\pi_\circ(\bd_0)$ in $\ccH_\circ$ 
corresponding to its smallest 
eigenvalue.
Note that Proposition \ref{rglVgqendv} and 
Proposition \ref{patmston} imply that 
$(\pi_\circ,\ccH_\circ)$ is uniquely determined by its highest weight.

Let $\widehat H_+\sseq \widehat H$ be the set of unitary characters of $H$ which appear as highest weights of irreducible positive energy  representations  
  of $G$. By the results of 
\cite{GoWa} and \cite{Toledano}, $\widehat H_+$ can be identified with the set of highest weights of unitarizable modules of the Virasoro algebra, whose description is by now well known 
\cite{GoKeO},   
  \cite[Sec. 3.3]{KaRa}, \cite{Langlands}. 
(Smoothness of the irreducible  positive energy representations is addressed in
\cite[p.499]{Toledano}.)   
    It follows that $\widehat{H}_+$ is a closed subset of $\widehat{H}$. In this paper we do not need the explicit description of $\widehat{H}_+$.

\begin{prp}
\label{P-sseqH+}
Let $(\pi,\ccH)$ be a positive energy representation of $G$, and let $\mathrm E^\pi$ be defined as in 
\eqref{dfn-Epi}. Then $\supp(\rmE^\pi)\sseq \widehat{H}_+$.
\end{prp}
\begin{proof}
By Proposition \ref{pipij} we can assume that $(\pi,\ccH)$ is of bounded type. 
Let $\cA\sseq \rmB(\ccW^\pi)$ denote the $C^*$-algebra generated by 
$\rho^\pi(H)$. By Remark \ref{rmk:i-iv}, 
$\rho$ has a unique extension  to a homomorphism of complex Lie groups $\rho^\pi:H_\C\to\GL(\ccW^\pi)$. Note that $\rho^\pi(H_\C)\sseq\cA$.

\textbf{Step 1.} By Proposition \ref{rglVgqendv} we can assume that $(\pi,\ccH)$ is embedded as a $G$-invariant Hilbert space in $\cC^\infty(G,\ccW^\pi)_{\rho^\pi,\beta_u}$. Let $K:\bS\times \oline \bS\to \rmB(\ccW^\pi)$ be the corresponding reproducing kernel. Next we prove that
$K(p,p')\in\cA$ for every $(p,p')\in \bS\times\oline \bS$. 
If $p = [g,s]$, then  $K(p,p)=\rho^\pi(s^{-1})\rho^\pi(s^{-1})^*\in\cA$.
Now let  $\mathbf q_\mathcal A^{}:\mathrm B(\ccW)\to \mathrm B(\ccW)/\mathcal A$ denote the canonical projection. 
The quotient $\mathrm B(\ccW)/\mathcal A$ is a complex Banach space.  
The map $\mathbf q_\mathcal A^{}\circ K:\bS\times \oline \bS\to 
\mathrm B(\ccW)/\mathcal A$ 
vanishes on the diagonal, and Proposition 
\ref{prpEVinBe}(iv) implies that $\mathbf q_\mathcal A^{}\circ K$
is holomorphic. Therefore 
$\mathbf q_\mathcal A^{}\circ K(p,p')$ vanishes 
for every $(p,p')\in \bS\times\oline \bS$. 
Consequently, $K(p,p')\in\mathcal A$ for every $(p,p')\in \bS\times\oline \bS$.

\textbf{Step 2.} The $C^*$-algebra of $H$ is isomorphic to $\cC_\circ(\whH)$, the algebra of complex-valued continuous functions on $\widehat{H}$ which vanish at infinity. Let $I^\pi$ denote the kernel of the canonical map
\begin{equation}
\label{C*map}
\cC_\circ(\whH)\to\rmB(\ccW^\pi)\ ,\ 
\psi\mapsto 
\int_{\whH}\psi(\chi)d\rmE^\pi(\chi).
\end{equation}
Then 
$
I^\pi=\{\psi\in \cC_\circ(\whH)\,:\,
\psi|^{}_{\supp(\rmE^\pi)}\equiv0
\}$. Since $\supp(\mathrm E^\pi)$ is compact,  
it follows that 
the image of 
\eqref{C*map} is a dense subspace of $\cA$, and 
 $\cC_\circ(\whH)/I^\pi\cong \cC(\supp(\rmE^\pi))$, where
$\cC(\supp(\rmE^\pi))$ is the space of continuous complex-valued functions on $\supp(E^\pi)$. 
By elementary $C^*$-theory
\cite[Sec. 1.8]{Dix}
 the map \eqref{C*map} 
induces an isomorphism 
$\iota^\pi:\cC(\supp(\rmE^\pi))\to \cA$ of $C^*$-algebras.

\textbf{Step 3.} Fix $\chi\in\supp(\rmE^\pi)$ and let
$\Lambda_\chi:\cA\to\C$ denote the character of $\cA$ 
defined by the relation $\Lambda_\chi\circ\iota^\pi(\psi)=\psi(\chi)$ for every $\psi\in \cC(\supp(\rmE^\pi))$. Note that
$\Lambda_\chi(\rho^\pi(h))=\chi(h)$ for every $h\in H$, and since both sides are holomorphic functions of $h\in H_\C$, equality holds for every $h\in H_\C$ as well.
Now set $K_\chi:=\Lambda_\chi\circ K$. It is clear that $K_\chi:\bS\times 
\oline \bS\to\C$ is holomorphic and $(G,\chi\circ\rho)$-invariant. 
Moreover, $K_\chi$ is positive definite, because every character $\Lambda_\chi$  is completely positive 
in the sense of \cite[Def. IV.3.3]{Takesaki}.
Consequently, by Proposition \ref{pEXGi},  
there exists a $G$-invariant Hilbert space 
$\ccH_\chi\sseq \cC^\infty(G,\C)_{\chi,\beta_u}$
with reproducing kernel $K_\chi$. 

\textbf{Step 4.} 
Let  $(\pi_\chi,\ccH_\chi)$ denote the smooth unitary
representation of $G$ on $\ccH_\chi$ that is defined in
Proposition \ref{prpEVinBe}(i). By the isomorphism 
\eqref{GaMinfC-H} we can assume $\ccH_\chi\sseq \mathcal O(\bS,\C)_{\chi\circ \rho}$.
Set $K_{\chi,p_\circ}(p):=K_\chi(p,p_\circ)\in\C$, so that $K_{\chi,p_\circ}\in\ccH$.
Note that 
$\ccH_\chi\neq\{0\}$, because 
$K_{\chi,p_\circ}(p_\circ)=\Lambda_\chi(K(p_\circ,p_\circ))=~1$. 
Now S.~Kobayashi's criterion for irreducibility of holomorphic induction \cite[Thm IV.1.10]{NeHo}
implies that $(\pi_\chi,\ccH_\chi)$ is irreducible (see also Proposition~\ref{adprp-Gi-ii-iii}).

It is straightforward to check that 
$\pi_\chi(h)K_{\chi,m_\circ}=\chi(h)K_{\chi,m_\circ}$ for every $h\in H$.
Therefore
Proposition \ref{adprp-Gi-ii-iii}(iv) implies that $\mathrm{im}(\ev_\yek^*)=
\mathrm{im}(\ev_{m_\circ}^*)=
\C K_{\chi,m_\circ}=(\ccH_\chi^\infty)^{\g n}$, and 
Proposition \ref{p(i)-(iv)}(iv) implies that 
$(\pi_\chi,\ccH_\chi)$ is a highest weight representation with highest weight $\chi$.
Thus from the definition of $\whH_+$ it follows that $\chi\in\whH_+$.
Now recall that in Step 3, $\chi$ is an arbitrary element of $\supp(\rmE^\pi)$. Consequently, $\supp(\rmE^\pi)\sseq \whH_+$.
\end{proof}
The next proposition is a converse to Proposition \ref{P-sseqH+}.

\begin{prp}
\label{adprpexGin}
Let $(\rho,\ccW)$  be a unitary representation of $H$, and $\rmE:\mathfrak B(\whH)\to \rmB(\ccW)$ denote the spectral measure corresponding to $\rho$. Assume that $\supp(\rmE)$ is a compact subset of $\whH_+$.  
Then there exists a
 $G$-invariant Hilbert space
$\ccH\sseq \cC^\infty(G,\ccW)_{\rho,\beta_u}$. 
\end{prp}
\begin{proof}
Let $\cA\sseq \mathrm B(\ccW)$ denote the $C^*$-algebra generated by $\rho(H)$. 
By an argument similar to Step~2 in the proof of Proposition \ref{P-sseqH+}, we obtain a canonical isomorphism of $C^*$-algebras
$\iota:\cC(\supp(\rmE))\to\cA$. For every $\chi\in\supp(\rmE)$, let $\Lambda_\chi\in\widehat\cA$ be the character of $\cA$ defined by 
$\Lambda_\chi\circ\iota(\psi)=\psi(\chi)$
for every $\psi\in \cC(\supp(\rmE))$, and let
$(\pi_\chi,\ccH_\chi)$ denote the irreducible highest weight representation corresponding to $\chi$.
Let $\ccW_\chi\sseq\ccH_\chi$ be the one-dimensional $\chi$-weight space of $\pi_\chi$.
Next we use a construction similar to the universal representation of a $C^*$-algebra. Set 
\[
\ccH_\cA:=\bigoplus_{\chi\in\supp(\rmE)}\ccH_\chi,\
\pi_\cA:=\bigoplus_{\chi\in\supp(\rmE)}\pi_\chi,\
\ccW_\cA:=\bigoplus_{\chi\in\supp(\rmE)}\ccW_\chi,\,
\text{ and }\,
\rho_\cA^{}:=\bigoplus_{\chi\in\supp(\rmE)}\chi,
\]
so that $(\pi_\cA,\ccH_\cA)$ is a 
representation of $G$, and $(\rho_\cA,\ccW_\cA)$ is a representation of $H$.
Set
\[
\widetilde\rho_\cA^{}:\cA\to\rmB(\ccW_\cA)\ ,\ \widetilde\rho_\cA^{}(a):=
\bigoplus_{\chi\in\supp(\rmE)}\Lambda_\chi(a).
\]
From the definition of $\Lambda_\chi$ it follows that $\Lambda_\chi(\rho(h))=\chi(h)$ for every $h\in H$, so that $\widetilde \rho_\cA^{}(\rho(h))=\rho_\cA^{}(h)$ for every $h\in H$. Moreover, 
$\widetilde \rho_\cA^{}$ is faithful because
$\widehat \cA\cong\supp(\rmE)$.
Consequently, 
elementary $C^*$-algebra theory implies that $\widetilde\rho_\cA^{}$ is an 
isomorphism of $C^*$-algebras between $\cA$ and the $C^*$-algebra $\widetilde A$ that is generated by $\rho_\cA^{}(H)$.

Since $\rho_\cA^{}$ is a bounded unitary representation of $H$, it has a unique holomorphic extension to $H_\C$ and by Proposition \ref{rglVgqendv}
we can embed $(\pi_\cA,\ccH_\cA)$  into $\cC^\infty(G,\ccW_\cA)_{\rho_\cA^{},\beta_u}$ as a $G$-invariant Hilbert space. 
Let $K_\cA:\bS\times\oline \bS\to\rmB(\ccW_\cA)$ be the $(G, \rho_\cA)$-invariant 
reproducing kernel of $\ccH_\cA$. An argument similar to  Step 1 
in the proof of Proposition \ref{P-sseqH+} shows that
$K_\cA(p,p')\in\widetilde \cA$ for every $(p,p')\in \bS\times\oline \bS$.
The map 
$(\widetilde\rho_\cA^{})^{-1}:\widetilde\cA\to \rmB(\ccW)$ is completely positive in the sense of 
\cite[Def. IV.3.3]{Takesaki}, and therefore
the kernel 
$K:\bS\times\oline \bS\to\rmB(\ccW)$, defined by $K(p,p')
:=(\widetilde\rho_\cA^{})^{-1}(K_\cA(p,p'))$,
is positive definite. Furthermore, the kernel $K$ is  
$(G,\rho)$-invariant, holomorphic, and  satisfies 
$K(p_\circ,p_\circ)=\yek_\ccW$.
Consequently, Proposition \ref{pEXGi} 
implies that there exists a $G$-invariant Hilbert space 
$\ccH\sseq \cC^\infty(G,\ccW)_{\rho,\beta_u}$ corresponding to $K$.
\end{proof}

\subsection{The direct integral $(\pi_\mu,\ccH_\mu)$}
\label{sec-dirint}
For every $\chi\in\whH_+$, let 
$(\pi_\chi,\ccH_\chi)$ denote the irreducible positive energy representation of $G$ with highest weight $\chi$. We will denote the inner product of $\ccH_\chi$ by 
$\lag\cdot,\cdot\rag_\chi$.
 Let $\mu$ be a finite positive Borel measure on $\whH_+$.
In this section we will construct a direct integral unitary representation 
\[
\pi_\mu:=\int^\oplus_{\whH_+}\pi_\chi d\mu(\chi).
\]
Throughout this section,  $(\rho_\mu,\ccW_\mu)$ will 
denote the cyclic unitary representation of $H$ corresponding to $\mu$, that is,
$\ccW_\mu:=L^2(\whH_+,d\mu)$ and  the action of 
$H$ is defined by 
\begin{equation}
\label{ad-dfrpcybd}
h\cdot f(\chi):=\chi(h)f(\chi)
\text{ \,for }
h\in H,\ \chi\in\whH_+,\ \text{and }f\in L^2(\whH_+,d\mu).
\end{equation}

By Proposition \ref{rglVgqendv}, the representation
$(\pi_\chi,\ccH_\chi)$ can be realized as a $G$-invariant Hilbert space in 
$\cC^\infty(G,\C)_{\chi,\beta_u}\cong \mathcal O(\bS,\C)_\chi$, and in fact Proposition \ref{patmston} implies that there exists a unique $(G,\chi)$-invariant holomorphic kernel $K_\chi:\bS\times \oline \bS\to\C$ satisfying $K_\chi(p_\circ,p_\circ)=\yek$.
Set $K_{\chi,p}(p'):=K_\chi(p',p)$, so that $\{K_{\chi,p}\,:\,p\in \bS\}\sseq\ccH_\chi$.

\begin{lem}
\label{LwHMMc}
The map $\whH_+\times \bS\times\oline \bS\to\C$, $(\chi,p,p')
\mapsto K_\chi(p,p')$ is continuous.
\end{lem}

\begin{proof}
It suffices to prove the statement for every subset of  the form $U\times 
\bS\times\oline \bS$, where $U\sseq \whH_+$ is open and relatively compact.
Choose a finite positive Borel measure $\mu$  on $\whH_+$ 
such that $\supp(\mu)=\oline U$.
By Proposition
\ref{adprpexGin} there exists a $G$-invariant Hilbert space 
$\ccH_{\rho_\mu}\sseq \cC^\infty(G,\ccW_\mu)_{\rho_\mu,\beta_u}$. 
Let 
\begin{equation}
\label{Krhmu-}
K_{\rho_\mu}:\bS\times\oline \bS\to\mathrm B(\ccW_\mu)
\end{equation} 
be the reproducing kernel of $\ccH_{\rho_\mu}$, as defined
in \eqref{RKMtM}.
By Proposition \ref{prpEVinBe}(iv), $K_{\rho_\mu}$ is holomorphic. Let $\cA_\mu\sseq\rmB(\ccW_\mu)$ denote the $C^*$-algebra generated by $\rho_\mu(H)$.
An argument similar to Step 1 in the proof of Proposition \ref{P-sseqH+} shows that $K_{\rho_\mu}(p,p')\in\cA_\mu$ for every $(p,p')\in \bS\times \oline \bS$. 
By an argument similar to Step 2 in the proof of 
Proposition \ref{P-sseqH+}, we obtain a canonical isomorphism of $C^*$-algebras 
$\iota^\mu:\cC(\supp(\mu))\to \cA_\mu$.
For every $\chi\in\supp(\mu)$, 
let $\Lambda_\chi:\cA_\mu\to\C$ be the character of $\cA_\mu$ that is defined by 
the relation 
$\Lambda_\chi\circ\iota^\mu(\psi)=\psi(\chi)$ for every $\psi\in \cC(\supp(\mu))$. 
An argument similar to Step 3 in the proof of 
Proposition \ref{P-sseqH+} shows that
$\Lambda_\chi\circ K_{\rho_\mu}$ is  holomorphic, $(G,\chi)$-invariant, 
and positive definite. Moreover,  $\Lambda_\chi\circ K_{\rho_\mu}(p_\circ,p_\circ)
=\yek$. Uniqueness of the reproducing kernel of $\ccH_\chi$ implies that 
\begin{equation}
\label{add-H+MolM}
K_\chi(p,p')=\Lambda_\chi(K_{\rho_\mu}(p,p'))
\text{ for every }\chi\in\supp(\mu)\text{ and every } (p,p')\in 
\bS\times \oline \bS.
\end{equation}
The statement of the lemma now follows from continuity of the canonical map 
$\cA_\mu\times\widehat{\cA}_\mu\to\C$.
\end{proof}

To construct the direct integral  
$
\pi_\mu=\int_{\whH_+}^\oplus\pi_\chi d\mu(\chi)
$,
we consider the family
$\{\ccH_\chi\, :\, \chi\in\whH_+\}$
as a field of Hilbert spaces  over $\whH_+$ (in the sense of \cite[Chap.~II]{Dix69}). Fix  a countable dense subset $\{p_n\}_{n=1}^\infty$ of $\bS$, and 
consider the
sections $\bfe_n\in\prod_{\chi\in \whH_+}\ccH_\chi$, defined by $\bfe_n(\chi):=K_{\chi,p_n}$ for every 
$\chi\in\whH_+$ and every $n\in\N$.
\begin{lem}
\label{add-lm-msrbity}
The sections $\bfe_n$, $n\in\N$, satisfy the following properties.
\begin{itemize}
\item[{\upshape(i)}] 
For every $\chi\in\whH_+$, the set $\{\bfe_n(\chi)\}_{n=1}^\infty$ is a total subset of $\ccH_\chi$.
\item[{\upshape(ii)}] For every $n,n'\in \N$, the map 
$
\whH_+\to\C
$, $\chi\mapsto\lag\bfe_n(\chi),\bfe_{n'}(\chi)\rag_\chi$,
is measurable.
\item[{\upshape(iii)}] For every $g\in G$, the operator field $\{\pi_\chi(g)\,:\,\chi\in\whH_+\}$ is measurable.
\end{itemize}
\end{lem}

\begin{proof}
(i) 
Recall that 
$\lag K_{\chi,p},K_{\chi,p'}\rag_\chi=K_\chi(p',p)$.
Thus, for every $n\in \N$, we have
\begin{align}
\label{K--K}
\lag K_{\chi,p}-K_{\chi,p_n},K_{\chi,p}-K_{\chi,p_n}\rag_\chi
=
K_\chi(p,p)+K_\chi(p_n,p_n)-K_\chi(p,p_n)-K_\chi(p_n,p).
\end{align}
Now let $\{p'_k\}_{k=1}^\infty$ be a subsequence of $\{p_{n}\}_{n=1}^\infty$
that converges to a point $p\in \bS$. 
Since $K_\chi$ is continuous, from \eqref{K--K} it follows that
$\lim_{k\to\infty}K_{\chi,p'_k}= K_{\chi,p}$ in $\ccH_\chi$. Furthermore, 
Proposition \ref{prpEVinBe}(i) states that
$\{K_{\chi,p}\, :\, p\in M\}$ is a total subset of $\ccH_\chi$. 
Therefore $\{K_{\chi,p_n}\}_{n=1}^\infty$ is a total subset of $\ccH_\chi$.

(ii) Follows from Lemma \ref{LwHMMc} because  $\lag\bfe_n(\chi),\bfe_{n'}(\chi)\rag_\chi=K_\chi(p_{n'},p_{n})$.

(iii) By $(G,\chi)$-invariance of $K_\chi$, we can write
$
\lag\pi_\chi(g)K_{\chi,p_n},K_{\chi,p_{n'}}\rag
=K_\chi(g^{-1}\cdot p_{n'},p_n)$.
Therefore, by Lemma \ref{LwHMMc}, the map
$\chi\mapsto \lag\pi_\chi(g)K_{\chi,p_n},K_{\chi,p_{n'}}\rag
$ is continuous.
The statement is now proved using a  
standard pointwise convergence argument
based on \cite[Chap.~1, \S 1, Lem. 1]{Dix}.
\end{proof}
We are now in a position to define $(\pi_\mu,\ccH_\mu)$. A section $\bfe\in \prod_{\chi\in\whH_+}\ccH_\chi$ is called \emph{measurable} if, for every $n\in\N$, the map 
\[
\whH_+\to\C\ ,\ 
\chi\mapsto \lag\bfe(\chi),\bfe_n(\chi)\rag_\chi
\]
is measurable. By standard facts in the theory of direct integral Hilbert spaces (see \cite{Dix} or \cite[Appendix II]{NeHo}),
Lemma \ref{add-lm-msrbity} implies that
the set $\ccH_\mu$, consisting of all  measurable sections $\bfe \in \prod_{\chi\in\whH_+}\ccH_\chi$ which satisfy 
$\int_{\whH_+}\|\bfe(\chi)\|_\chi^2d\mu(\chi)<\infty$,
is a separable Hilbert space with inner product
$\lag \bfe,\bfe'\rag:=\int_{\whH_+}\lag \bfe(\chi),\bfe'(\chi)\rag_\chi d\mu(\chi)$.
Setting 
\[
\left(\pi_\mu(g)\bfe\right)(\chi):=\pi_\chi(g)\bfe(\chi)
\text{ for $g\in G$, $\bfe\in\ccH_\mu$, and $\chi\in\whH_+$},
\] 
we obtain a representation of  
$G$ on $\ccH_\mu$ by unitary operators (but we still need to prove that $\pi_\mu$ is a smooth unitary representation of positive energy).

\begin{rmk}
\label{rmkCktnK}
Let $\whH_+=\bigcup_{n=1}^\infty X_n$
be a  partition of $\whH_+$ into measurable and 
relatively compact subsets, and set $\mu_n(\Omega):=\mu(\Omega\cap X_n)$
for every $\Omega\in\mathfrak B(\whH_+)$ and every $n\in\N$. Then
$
\mu=\sum_{n=1}^\infty\mu_n
$, and the $\mu_n$ are mutually singular measures with compact support. 
Moreover, 
$\pi_\mu\cong\bigoplus_{n=1}^\infty\pi_{\mu_n^{}}$.
\end{rmk}

\begin{prp}
\label{piMupE}
Let $\mu$ be a finite positive Borel measure on $\whH_+$. Then $(\pi_\mu,\ccH_\mu)$ is a positive energy representation of $G$.
\end{prp}
\begin{proof}
Because of Remark \ref{rmkCktnK}, we can assume
that
$\supp(\mu)$ is compact.
Let $D_\circ\cong\R $ be the one-parameter subgroup 
$e^{\R i\bd_0}\sseq G$.
Using
\cite[Thm X.6.16]{NeHo} for the restriction $\pi_\mu\big|_{D_\circ}$,
we obtain $\Spec(-\dd\pi_\mu(\bd_0))\sseq
\R^+\cup\{0\}$. 
Next we   
prove that $(\pi_\mu,\ccH_\mu)$ is a smooth unitary representation of $G$.   Given 
$f\in L^\infty(\widehat{H}_+,\mu)$ and $m\in M$, we set 
\begin{equation}
\label{Efmdf}
\bfe^{}_{f,p}\in \prod_{\chi\in\whH_+}\ccH_\chi\ ,\ \bfe_{f,p}(\chi):=f(\chi)K_{\chi,p}.
\end{equation}
Lemma \ref{LwHMMc} implies that
$\bfe^{}_{f,p}\in\ccH_\mu$, and 
\cite[Prop. II.1.7]{Dix69} implies that the fields $\bfe_{f,p}^{}$ 
generate a dense subspace of $\ccH_\mu$. For every $g_\circ\in G$, we have 
\begin{align*}
\lag \pi_\mu(g_\circ)\bfe_{f,p}^{},\bfe_{f,p}^{}\rag
&
=
\int_{\whH_+}f(\chi)
\lag\pi_\chi(g_\circ)K_{\chi,p},K_{\chi,p}\rag_\chi\,d\mu(\chi)
=\int_{\whH_+}f(\chi)K_\chi(g_\circ^{-1}\cdot p, p)\,d\mu(\chi).
\end{align*}  
The map $G\to\C$,
$g_\circ\mapsto K_\chi(g_\circ^{-1}\cdot p,p)$ is smooth. 
From \eqref{add-H+MolM}
we obtain
\begin{equation}
\label{ad-pmumfm}
|K_\chi(g_\circ^{-1}\cdot p,p)|
=|\Lambda_\chi(K_{\rho_\mu}(g_\circ^{-1}\cdot p,p))|
\leq 
\|K_{\rho_\mu}(g_\circ^{-1}\cdot p,p)\|,
\end{equation}
where $K_{\rho_\mu^{}}$ is defined as in 
\eqref{Krhmu-}.
From linearity of $\Lambda_\chi$ it follows that an  inequality similar 
to \eqref{ad-pmumfm} holds for 
directional derivatives of $K_\chi(g_\circ^{-1}\cdot p,p)$. 
By  differentiating under the integral sign   
(\cite[Lem. 2.1]{NSZ14})
it follows that the map $g_\circ\mapsto \lag \pi_\mu(g_\circ)\bfe_{f,p},\bfe_{f,p}\rag$ is smooth. Finally, from \cite[Thm 7.2]{nediff} 
 it follows that $\bfe_{f,p}$ is a smooth vector for $(\pi_\mu,\ccH_\mu)$.
\end{proof}

\begin{prp}
\label{adWolueH}
Let $\mu$ be a finite positive Borel measure on $\whH_+$. Let $\ccW^{\pi_\mu}:=\oline{(\ccH_\mu^\infty)^{\g n}}$ and set 
$\rho^{\pi_\mu}(h):=\pi_\mu(h)\big|_\ccW$ for every $h\in H$. Then 
$(\rho^{\pi_\mu},\ccW^{\pi_\mu})$ and $(\rho_\mu,\ccW_\mu)$ are unitarily equivalent. 
\end{prp}
\begin{proof}
By Remark \ref{rmkCktnK} we can assume that $\supp(\mu)$ is compact.
By the proof of Proposition \ref{piMupE}, the sections $\bfe_{f,p}^{}$, 
defined in \eqref{Efmdf}, belong to $\ccH^\infty_\mu$.  Applying \cite[Lem. 2]{Arnal} to 
one-parameter subgroups of $G$ implies that 
$\bfe_{f,p_\circ}^{}\in(\ccH^\infty_\mu)^{\g n}$.
Therefore, to complete the proof of the proposition, it suffices to show  
that if a section $\bfe\in\ccH_\mu^\infty$ satisfies $\bfe\in(\ccH^\infty_\mu)^{\g n}$, 
then 
$\bfe(\chi)\in\C K_{\chi,p_\circ}$ 
for almost every $\chi\in\whH_+$. 
To prove the latter claim, again we note that 
\cite[Lem. 2]{Arnal} implies that 
there exists a measurable set $N\sseq \whH_+$ with
$\mu(N)=0$ such that $\dd\pi_\chi(x)\bfe(\chi)=0$ for every $x\in\g n$ and every $\chi\in\whH_+\bls N$.
(Here we use the fact that for every $v\in\ccH^\infty_\mu$, the map $\gC\to\ccH$, $x\mapsto \dd\pi_\mu(x)v$ is continuous.)
Finally, Proposition \ref{rglVgqendv}
implies that $(\ccH_\chi^{\infty})^{\g n}=\C K_{\chi,p_\circ}$.
%
\end{proof}

The next theorem is our main result. It provides a complete characterization of positive energy representations of $G$.

\begin{thm}
\label{thm-MAIN}
Let $(\pi,\ccH)$ be a positive energy representation of $G$. Then 
\begin{equation}
\label{adpim12n}
\pi\cong \pi_{\mu_1^{}}\oplus 2\pi_{\mu_2^{}}
\oplus 3\pi_{\mu_3^{}}\oplus\cdots\oplus\infty\,\pi_{\mu_{\infty}^{}}
\end{equation}
where $\mu_1,\mu_2,\ldots,\mu_{\infty}$ are  mutually singular  finite positive measures on $\whH_+$, and the $\pi_{\mu_n^{}}$ are constructed as in 
\emph{Section \ref{sec-dirint}}.
Moreover, the measures $\mu_n, n\in\N\cup\{\infty\}$, are unique up to equivalence.
\end{thm}

\begin{proof}
Let
$\ccW^\pi:=\oline{(\ccH^\infty)^{\g n}}$ and
set $\rho^\pi(h):=\pi(h)\big|_\ccW$ for every $h\in H$.
First we prove the 
existence of the direct sum decomposition \eqref{adpim12n}.

\textbf{Step 1.} Assume that $(\rho^\pi,\ccW^\pi)$ is a cyclic representation of $H$, and the spectral 
measure $\mathrm E^\pi$, defined in \eqref{dfn-Epi}, has compact support. Then 
there exists a finite positive Borel measure with compact support $\mu$ on $\whH_+$ such that $(\rho^\pi,\ccW^\pi)\cong(\rho_\mu,\ccW_\mu)$. 
Proposition \ref{prpPoScyq} and Proposition \ref{rglVgqendv} imply that 
 $(\pi,\ccH)$ can be realized as a $G$-invariant Hilbert space 
 in $\cC^\infty(G,\ccW_\mu)_{\rho_\mu,\beta_u}$.
For a similar reason,  
 $(\pi_\mu,\ccH_\mu)$ can also be realized as a $G$-invariant Hilbert space in 
 $\cC^\infty(G,\ccW_\mu)_{\rho_\mu,\beta_u}$.
Therefore by Proposition \ref{patmston} we obtain
$(\pi,\ccH)\cong(\pi_\mu,\ccH_\mu)$.

\textbf{Step 2.} Assume that $(\rho^\pi,\ccW^\pi)$ is not necessarily cyclic, but
the support of $\mathrm E^\pi$ is compact. 
Then \[
\rho^\pi\cong\rho_{\mu_1^{}}
\oplus 
2\rho_{\mu_2^{}}
\oplus\cdots
\oplus
\infty\,\rho_{\mu_\infty^{}},
\]
where the $\mu_n$ are mutually singular finite positive Borel measures on $\whH_+$. By
Proposition \ref{adprp-Gi-ii-iii}(v), the above decomposition of $\rho^\pi$ corresponds to a decomposition
\[
\pi=\pi_1\oplus 2\pi_2\oplus\cdots\oplus
\infty\,\pi_\infty
\]
where $(\ccH_k^\infty)^\g n=\oline{(\ccH_k^\infty)^\g n}=\ccW_{\mu_k^{}}$ for $k\in\N\cup\{\infty\}$ (here $\ccH_k$ is the Hilbert space of $\pi_k$, and $\ccW_{\mu_k^{}}$ is the Hilbert space of 
$\rho_{\mu_k^{}}$). Since the representations $(\rho_{\mu_k^{}},\ccW_{\mu_k^{}})$ are cyclic, 
 Step 1 implies that $(\pi_k,\ccH_k)\cong(\pi_{\mu_k^{}},\ccH_{\mu_k^{}})$ for every $k\in\N\cup\{\infty\}$. This results in the
  decomposition \eqref{adpim12n} for $(\pi,\ccH)$.\\
  
 \textbf{Step 3.} Finally, note that Proposition \ref{pipij} together with Step 2 imply the existence of the direct sum decomposition \eqref{adpim12n} without any assumptions on $(\rho^\pi,\ccW^\pi)$.

Next we prove uniqueness of the decomposition
\eqref{adpim12n}. Suppose that 
a decomposition of $(\pi,\ccH)$ as in 
\eqref{adpim12n} is given. 
By considering $\g n$-invariants, and using   Proposition 
\ref{adWolueH},
 we obtain a corresponding direct sum decomposition 
\begin{equation}
\label{eQQrhoPId}
\rho^\pi=\rho_{\mu_1^{}}\oplus 2\rho_{\mu_2^{}}\oplus
\cdots\oplus
\infty\,\rho_{\mu_\infty^{}},
\end{equation}
where the $\mu_n$ are mutually singular finite positive Borel measures.
Since $H$ is a locally compact abelian group,  the decomposition \eqref{eQQrhoPId} 
 is unique up to equivalence of measures. Therefore the decomposition of $(\pi,\ccH)$ is also unique.
\end{proof}

\end{document}